\documentclass[11pt]{amsart}
\usepackage[T1]{fontenc}
\usepackage{lmodern,amsmath,amssymb,mathtools,microtype,esint}
\usepackage[margin=1.05in]{geometry}
\usepackage{enumitem,xcolor,etoolbox}
\usepackage{booktabs,tabularx,placeins}
\usepackage[colorlinks=true,linkcolor=blue,citecolor=blue,urlcolor=blue]{hyperref}
\makeatletter
\patchcmd{\@setauthors}{\MakeUppercase{\authors}}{\authors}{}{}
\makeatother
\setlist[enumerate]{label=\textup{(\roman*)},leftmargin=2.1em}
\numberwithin{equation}{section}
\newtheorem{theorem}{Theorem}[section]
\newtheorem{lemma}[theorem]{Lemma}
\newtheorem{proposition}[theorem]{Proposition}
\newtheorem{corollary}[theorem]{Corollary}
\theoremstyle{definition}\newtheorem{example}[theorem]{Example}
\theoremstyle{remark}
\newcommand{\C}{\mathbb C}\newcommand{\R}{\mathbb R}
\newcommand{\IDA}{\mathrm{IDA}}\newcommand{\BDA}{\mathrm{BDA}}\newcommand{\VDA}{\mathrm{VDA}}
\newcommand{\ess}{\mathrm{ess}}\newcommand{\Sch}{\mathcal S}
\newcommand{\Symb}{\mathcal D_\varphi}
\newcommand{\norm}[1]{\left\lVert #1\right\rVert}
\newcommand{\avint}{\mathop{\fint}}
\DeclareMathOperator{\spanop}{span}\DeclareMathOperator{\supp}{supp}
\title[IDA symbols and Toeplitz operators]{IDA symbols and Schatten--Lorentz theory for Toeplitz operators on weighted Fock spaces}
\author[Chunxu Xu and Jianxiang Dong]{Chunxu Xu\textsuperscript{a},
Jianxiang Dong\textsuperscript{b,*}}
\dedicatory{\textsuperscript{a}School of Science, Nanjing Forestry University,
Nanjing 210037, P.R. China\\
\textsuperscript{b}School of Mathematics and Statistics, Tianshui Normal University,
Tianshui 741000, P.R. China}
\thanks{\textsuperscript{*}Corresponding author.}
\thanks{\textit{Email addresses:}
\href{mailto:1968385450@qq.com}{1968385450@qq.com} (Chunxu Xu),
\href{mailto:jianxd@tsnu.edu.cn}{jianxd@tsnu.edu.cn} (Jianxiang Dong).}
\hypersetup{
 pdftitle={IDA symbols and Schatten--Lorentz theory for Toeplitz operators on weighted Fock spaces},
 pdfauthor={Chunxu Xu and Jianxiang Dong},
 pdfsubject={Weighted Fock spaces, IDA symbols, Schatten--Lorentz ideals, and Weyl asymptotics},
 pdfkeywords={Toeplitz operator, IDA space, weighted Fock space, Schatten--Lorentz ideal, singular values, Weyl asymptotics}
}
\date{}
\newcommand{\articlekeywords}{Weighted Fock space, Toeplitz operator, IDA space,
Berezin transform, Schatten--Lorentz ideal, singular values, Weyl asymptotics}
\newcommand{\articlemscyear}{2020}
\newcommand{\articlemsc}{Primary 47B35; Secondary 32A25, 32A37, 35P20, 47B10}
\AtBeginDocument{%
}
\begin{document}
\raggedbottom
\begin{abstract}
We study complex-symbol Toeplitz operators on weighted Fock spaces.
The weight has a uniformly positive and bounded real Hessian.
Local analytic distance is unchanged when an entire function is added
to the symbol. We combine this distance with a complex ball average or
the Berezin transform to control the local size of the symbol.
Under a matching Lorentz condition on the distance, we characterize
membership in every Schatten--Lorentz
class. The local symbol exponent can be any $1\le q<\infty$ and is independent
of the two Schatten--Lorentz indices. The proof uses positive Toeplitz estimates,
Hankel estimates, and a local decomposition that preserves the initial
domain. Two examples show why the local exponent and the secondary
Lorentz index must be kept. We also obtain mixed mapping criteria and
regular singular-value decay estimates. An error operator with a
lower-order counting function does not change the leading coefficient. This gives
explicit Weyl constants for quadratic weights. A radial example shows
that uniform Hessian bounds alone do not imply such a constant.
\end{abstract}
\maketitle
\enlargethispage{3pt}

\begingroup
\small
\noindent\textit{Keywords.}\enspace \articlekeywords.

\smallskip
\noindent\textit{\articlemscyear\ Mathematics Subject Classification.}\enspace
\articlemsc.
\endgroup

\medskip
\setcounter{tocdepth}{1}
\tableofcontents
\medskip

\section{Introduction}
Toeplitz operators with positive symbols or positive measure symbols
are closely related to local averages and Carleson-type conditions.
Such criteria are well developed on classical and generalized Fock
spaces, including mappings between spaces with different exponents;
see \cite{IZ,HL2011,HL,IVW}. For complex symbols, cancellation makes
the problem different. Berezin transforms and mean oscillation provide
useful tools, in particular for BMO-type symbols \cite{BCI,CIL}.

A different approach is based on local distance to holomorphic
functions. This idea goes back to Luecking's work on Bergman Hankel
operators \cite{Luecking}. Hu and Virtanen introduced the IDA spaces
in the Fock setting and used them to characterize Hankel operators
\cite{HVAPDE,HVTrans}. Related BMO and mean-oscillation criteria for
simultaneous boundedness or compactness of \(H_f\) and \(H_{\bar f}\)
appear in \cite{Bauer,HuWang}. Complex-symbol Toeplitz mappings with
BMO and IMO-type symbols were studied in
\cite{WangBMOIMO,ZhangCaoHe}. Compactness of bounded-symbol Hankel
operators on classical \(F^p\) spaces, \(1<p<\infty\), was also studied
in \cite{HaV}.

Local analytic distance is unchanged when an entire function is added
to a symbol. This addition can change the Toeplitz operator.
For example, \(f(z)=z_1\) has zero local analytic
distance, whereas \(T_f\) is unbounded on the classical Fock spaces;
see Example~\ref{ex:analytic-part}.
This leads to two questions. Which scalar conditions recover the
information missed by local analytic distance? How do the resulting
criteria depend on the Fock space exponents and the local symbol exponent?

We use a complex ball average or the Berezin transform as the scalar
test. The mean-value property explains the role of the ball average.
For a holomorphic function, this average equals its value at the centre.
It therefore detects information that analytic distance misses.
For a complex symbol, cancellation prevents the Berezin transform
alone from controlling local size. Under the stated IDA hypotheses,
we prove that either scalar test controls local size when combined
with analytic distance. The resulting criteria are stated in
Theorem~\ref{thm:lorentz} for Schatten--Lorentz ideals and in
Theorem~\ref{thm:mixed} for mixed mappings. Both start from the natural
domain spanned by reproducing kernels. The scalar estimates are
proved before boundedness of \(T_f\) is known.

Let $\varphi\in C^2(\C^n;\R)$ satisfy
\begin{equation}\label{eq:weight}
 m_\varphi I_{2n}\le\operatorname{Hess}_{\R}\varphi(z)
 \le M_\varphi I_{2n},
 \qquad z\in\C^n,\quad 0<m_\varphi\le M_\varphi<\infty.
\end{equation}
This real Hessian condition implies the complex Hessian condition
used in \cite{HVAPDE}. Hence all kernel and local estimates quoted
from that paper apply here. Condition \eqref{eq:weight} includes
$\varphi(z)=\alpha|z|^2/2$, $\alpha>0$. It also includes nonradial
weights such as $\varphi(z)=|z|^2+\sin(\operatorname{Re}z_1)$.
For $1\le p<\infty$, write
\[
 L^p_\varphi=L^p(\C^n,e^{-p\varphi}dv),\qquad
 F^p_\varphi=L^p_\varphi\cap\mathcal H(\C^n),\qquad
 \norm{g}_{p,\varphi}=\norm{ge^{-\varphi}}_{L^p(dv)}.
\]
Here $dv$ is Lebesgue measure. We write $\mathcal H(\Omega)$ for the
holomorphic functions on $\Omega$. The classical $L^p$ Fock theory goes
back to \cite{JPR}.

Let $P$ and $K$ be the projection and kernel of $F^2_\varphi$. Put
\[
 k_z(w)=K(w,z)/K(z,z)^{1/2},\qquad
 \Gamma=\spanop\{k_z:z\in\C^n\}.
\]
Following \cite[(2-7)]{HVAPDE}, use the initial symbol domain
\begin{equation}\label{eq:domain}
 \mathcal D_\varphi=\{f\text{ measurable}:fg\in L^1_\varphi
                 \text{ for every }g\in\Gamma\}.
\end{equation}
The condition $f\in\mathcal D_\varphi$ ensures that the following
operators are defined on $\Gamma$.
The projection extends boundedly to $L^1_\varphi$. Thus, on $\Gamma$,
\[
 T_fg=P(fg),\qquad H_fg=(I-P)(fg),\qquad M_fg=fg.
\]
We always mean extensions of these initial operators. The notation $\norm{T_f}_{p\to q}$ refers to $T_f:F^p_\varphi\to F^q_\varphi$.

For \(z\in\mathbb C^n\) and \(r>0\), let
\(B(z,r)=\{w\in\mathbb C^n:|w-z|<r\}\), and let
\(\avint_E\) denote the normalized Lebesgue average over \(E\).
For \(q>0\), \(r>0\), and \(f\in L^q_{\rm loc}(\mathbb C^n)\), set
\[
\begin{aligned}
G_{q,r}f(z)
&=
\inf_{h\in\mathcal H(B(z,r))}
\left(
\avint_{B(z,r)} |f-h|^q\,dv
\right)^{1/q},\\
M_{q,r}f(z)
&=
\left(
\avint_{B(z,r)} |f|^q\,dv
\right)^{1/q}.
\end{aligned}
\]
Here \(G_{q,r}f(z)\) is the local \(L^q\)-distance from \(f\) to
holomorphic functions on \(B(z,r)\), while \(M_{q,r}f(z)\) is the
local \(L^q\)-mean of \(f\).
For $0<s\le\infty$, define
\[
 \IDA^{s,q}=\{f\in L^q_{\rm loc}:G_{q,1}f\in L^s(dv)\},\qquad
 \norm{f}_{\IDA^{s,q}}=\norm{G_{q,1}f}_{L^s(dv)}.
\]
Write $\BDA^q=\IDA^{\infty,q}$. The class $\VDA^q$ consists of
the functions in $L^q_{\rm loc}$ for which $G_{q,1}f(z)\to0$ as
$|z|\to\infty$. These quantities are seminorms when $s,q\ge1$
and quasi-seminorms in general. Adding an entire function leaves
$G_{q,r}f$ unchanged. For $q\ge1$, changing the fixed radius gives
equivalent seminorms and the same VDA class \cite[Section 3]{HVAPDE}.

For $q\ge1$ we use two scalar functions:
\[
 a_rf(z)=\avint_{B(z,r)}f\,dv,\qquad
 b_f(z)=\int_{\C^n}f(w)|k_z(w)|^2e^{-2\varphi(w)}\,dv(w).
\]
The first function is a complex average. The second is the Berezin
transform. It equals $T_fk_z(z)/K(z,z)^{1/2}$. The integral converges
absolutely because $fk_z\in L^1_\varphi$ and $|k_z|e^{-\varphi}$
is bounded. In the next theorem, the same exponent $q$ is used in
the range space $F^q_\varphi$ and in the local means of the symbol.

\begin{theorem}[Different Fock exponents]\label{thm:mixed}
Assume \eqref{eq:weight}, let $1\le p,q<\infty$, let $f\in\Symb\cap L^q_{\rm loc}$, and fix $r>0$.
\begin{enumerate}
\item Suppose $p\le q$ and $f\in\BDA^q$. Then
\[
 T_f:F^p_\varphi\to F^q_\varphi\text{ is bounded}
 \ \Longleftrightarrow\ b_f\in L^\infty
 \ \Longleftrightarrow\ a_rf\in L^\infty
 \ \Longleftrightarrow\ M_{q,r}f\in L^\infty.
\]
Moreover,
\begin{equation}\label{eq:boundednorm}
 \norm{T_f}_{p\to q}+\norm{G_{q,r}f}_\infty
 \asymp\norm{M_{q,r}f}_\infty
 \asymp\norm{b_f}_\infty+\norm{G_{q,r}f}_\infty
 \asymp\norm{a_rf}_\infty+\norm{G_{q,r}f}_\infty.
\end{equation}
\item Suppose $p\le q$ and $f\in\VDA^q$. Then
\[
 T_f:F^p_\varphi\to F^q_\varphi\text{ is compact}
 \ \Longleftrightarrow\ b_f(z)\to0
 \ \Longleftrightarrow\ a_rf(z)\to0
 \ \Longleftrightarrow\ M_{q,r}f(z)\to0.
\]
All limits are as $|z|\to\infty$. If $T_f$ is bounded, then
\begin{equation}\label{eq:ess}
 \norm{T_f}_{\ess,p\to q}
 \asymp\limsup_{|z|\to\infty}|b_f(z)|
 \asymp\limsup_{|z|\to\infty}|a_rf(z)|
 \asymp\limsup_{|z|\to\infty}M_{q,r}f(z).
\end{equation}
\item Suppose $q<p$, put $s=pq/(p-q)$, and assume $f\in\IDA^{s,q}$. Then
\begin{equation*}
 \begin{split}
 T_f:F^p_\varphi\to F^q_\varphi\text{ is bounded}
 &\ \Longleftrightarrow\ T_f\text{ is compact}\\
 &\ \Longleftrightarrow\ b_f\in L^s
 \ \Longleftrightarrow\ a_rf\in L^s
 \ \Longleftrightarrow\ M_{q,r}f\in L^s.
 \end{split}
\end{equation*}
The corresponding norm estimate is
\begin{equation}\label{eq:descendingnorm}
 \norm{T_f}_{p\to q}+\norm{G_{q,r}f}_s
 \asymp\norm{M_{q,r}f}_s
 \asymp\norm{b_f}_s+\norm{G_{q,r}f}_s
 \asymp\norm{a_rf}_s+\norm{G_{q,r}f}_s.
\end{equation}
\end{enumerate}
Constants may depend on the fixed weight bounds, dimension, radius, and exponents, but not on $f$.
\end{theorem}

The three parts use different IDA hypotheses. The Hankel theorem in
\cite{HVAPDE} controls $H_f$. The Carleson measure theorem in
\cite{HL} controls multiplication by $f$. We estimate $M_{q,r}f$
from $b_f$ or $a_rf$, together with $G_{q,r}f$. We prove these
estimates before boundedness of $T_f$ is known. For $p>q$, we also prove an
$L^{pq/(p-q)}$ estimate for the Berezin diagonal of every bounded map
$F^p_\varphi\to F^q_\varphi$.

For Schatten--Lorentz operators on $F^2_\varphi$, the local symbol
exponent need not equal $2$.
We use the full two-parameter Lorentz scale. Let \(v^*\) denote the
decreasing rearrangement of \(|v|\). For
\(0<\rho<\infty\) and \(0<\tau<\infty\), set
\[
 \|v\|_{L^{\rho,\tau}}
 =
 \left(
  \int_0^\infty
  \bigl[t^{1/\rho}v^*(t)\bigr]^\tau
  \,\frac{dt}{t}
 \right)^{1/\tau},
\]
and, when \(\tau=\infty\),
\[
 \|v\|_{L^{\rho,\infty}}
 =
 \sup_{t>0} t^{1/\rho}v^*(t).
\]
The corresponding Lorentz spaces are denoted by
\(L^{\rho,\tau}\).

For a sequence \(x=(x_k)_{k\ge1}\), let
\(x_1^*\ge x_2^*\ge\cdots\ge0\) be the decreasing rearrangement of
\((|x_k|)\), and define its associated step-function rearrangement by
\[
 x^*(t)=x_k^*,
 \qquad k-1\le t<k,\quad k=1,2,\ldots .
\]
The Lorentz sequence space \(\ell^{\rho,\tau}\) is defined by
\[
 \|x\|_{\ell^{\rho,\tau}}
 =
 \left(
  \int_0^\infty
  \bigl[t^{1/\rho}x^*(t)\bigr]^\tau
  \,\frac{dt}{t}
 \right)^{1/\tau},
 \qquad 0<\tau<\infty,
\]
and
\[
 \|x\|_{\ell^{\rho,\infty}}
 =
 \sup_{t>0} t^{1/\rho}x^*(t).
\]

If \(A\) is a compact operator between Hilbert spaces, we write
\(A\in\Sch_{\rho,\tau}\) when its decreasing singular-value sequence
\((s_k(A))_{k\ge1}\) belongs to \(\ell^{\rho,\tau}\), and set
$
 \|A\|_{\Sch_{\rho,\tau}}
 =
 \|(s_k(A))_{k\ge1}\|_{\ell^{\rho,\tau}}.
$
Thus \(L^{\rho,\rho}=L^\rho\) and
\(\Sch_{\rho,\rho}=\Sch_\rho\), with equivalent normalizations, while
\(\Sch_{\rho,\infty}\) is the weak Schatten class.
We also write
$N(t;A)=\#\{k:s_k(A)>t\}$ and define
\[
 \IDA^{(\rho,\tau),q}
 =\{f\in L^q_{\rm loc}:G_{q,1}f\in L^{\rho,\tau}\}.
\]
Lemma~\ref{lem:lorentz-radius} below shows that the radius $1$ may be
replaced by any fixed positive radius.

\begin{theorem}[The full two-parameter Schatten--Lorentz scale]
\label{thm:lorentz}
Assume \eqref{eq:weight}. Let $1\le q<\infty$,
$0<\rho<\infty$, $0<\tau\le\infty$, and $r>0$. Suppose that
$f\in\Symb\cap L^q_{\rm loc}$ and
$G_{q,r}f\in L^{\rho,\tau}$; equivalently,
$f\in\Symb\cap\IDA^{(\rho,\tau),q}$. Then, on $F^2_\varphi$, the
following conditions are equivalent:
\begin{enumerate}
\item $T_f\in\Sch_{\rho,\tau}$;
\item $b_f\in L^{\rho,\tau}$;
\item $a_rf\in L^{\rho,\tau}$;
\item $M_{q,r}f\in L^{\rho,\tau}$.
\end{enumerate}
If $D_{\rho,\tau}=\|G_{q,r}f\|_{\rho,\tau}$, then
\begin{equation}\label{eq:lorentznorm}
 \begin{split}
 \|T_f\|_{\Sch_{\rho,\tau}}+D_{\rho,\tau}
 &\asymp \|M_{q,r}f\|_{\rho,\tau}\\
 &\asymp \|a_rf\|_{\rho,\tau}+D_{\rho,\tau}
 \asymp \|b_f\|_{\rho,\tau}+D_{\rho,\tau}.
 \end{split}
\end{equation}
The constants may depend on the displayed parameters but not on $f$.
\end{theorem}

The theorem requires $G_{q,r}f\in L^{\rho,\tau}$ with the same
indices as the operator ideal. The key local-size estimate is
\[
 \|M_{q,r}f\|_{\rho,\tau}
 \le C\bigl(\|b_f\|_{\rho,\tau}
     +\|G_{q,r}f\|_{\rho,\tau}\bigr).
\]
This estimate is obtained before boundedness of \(T_f\) is known.
The proof uses local holomorphic approximants to control the symbol
by its scalar data and analytic distance.

The three indices have different roles. The exponent \(q\) measures
local integrability of the symbol. The indices \(\rho\) and \(\tau\)
specify the Lorentz norm of the scalar data and of the singular-value
sequence. For a fixed operator ideal, we may choose the local exponent
\(q\), provided the corresponding Lorentz--IDA hypothesis holds.

This matters especially when \(1\le q<2\). A local \(L^q\) condition
does not imply local square integrability. Nor does a Lorentz bound for
\(G_{q,r}f\) imply the corresponding bound for \(G_{2,r}f\).
The local-\(L^2\) Hankel ideal criteria in \cite{HVTrans,FWZ}
therefore cannot be applied directly to \(f\) under our hypotheses.
Lemma~\ref{lem:q-decomposition} resolves this difficulty. It constructs
\(f=g+h\), where \(g\) satisfies the required local-\(L^2\) IDA
estimate and \(gk_z\in L^2_\varphi\) for every \(z\).
The local \(L^q\) size of \(h\) is controlled by the original analytic
distance. Positive Toeplitz estimates then control \(T_h\) in the same
ideal. Both parts preserve the initial domain. Thus the Hilbert-space
argument applies to \(g\), while \(f\) needs only the stated local
\(L^q\) hypothesis.

The local exponent matters. The symbol in
Example~\ref{ex:local-q-below-two} belongs to
\(\IDA^{(\rho,\tau),q}\) for every
\(1\le q<2\), \(0<\rho<\infty\), and \(0<\tau\le\infty\),
but belongs to no \(\IDA^{(\rho,\infty),2}\) with
\(0<\rho<\infty\).

The cases $\tau=\rho$ and $\tau=\infty$ give the strong and weak
Schatten criteria. The analytic-distance hypothesis is essential.
For $0<\tau<\rho$, the weaker condition $G_{q,r}f\in L^\rho$
need not suffice. Example~\ref{ex:secondary-index} shows that the
equivalence can then fail. This happens even for $q=2$ and a bounded
smooth symbol on classical Fock space. The proof uses small powers
to pass to Banach Lorentz spaces. It then applies Hardy inequalities.
No interpolation identity for IDA spaces is needed.

\subsection*{Relation to earlier operator-ideal and rearrangement criteria}

For positive Toeplitz operators, local masses and Berezin transforms
give ideal criteria; see \cite{IVW,Orenstein,HWZ}. Schatten and related ideal
criteria for Hankel operators are developed in
\cite{Farnsworth,HVTrans,HVAPDE,XuDoubling,XuAMS,FWZ}. These results
provide the positive and Hankel estimates used below. For a complex
Toeplitz symbol, we also need to control the local size that analytic
distance alone does not detect.

Fan, Wang, and Zeng \cite{FWZ} treat rearrangement estimates for an
$L^2$-based IDA function in a broader weighted Bergman--Fock framework.
These Hankel criteria depend on analytic distance. Our Toeplitz
criteria also use a complex average or the Berezin transform.
We obtain them for an independent local exponent in several variables
under \eqref{eq:weight}. Huang, Wang, and Zeng
\cite{HWZ} treat several-variable radial Fock-type spaces with
adapted regions. They estimate partial sums of positive Toeplitz
eigenvalues and prove Hankel form bounds. We prove the fixed-scale estimates needed here from the
positive theorem in \cite{IVW} and the decomposition in \cite{HVAPDE}.
Thus Lemmas~\ref{lem:positive-lorentz} and
\ref{lem:lorentz-hankel} serve as auxiliary estimates under
\eqref{eq:weight}. We do not claim these estimates as separate new results.

Our main contribution is to control local symbol size by $a_rf$ or
$b_f$, together with analytic distance. The local exponent is
independent of the ideal indices. The decomposition also preserves
the initial domain. The same estimates give regular singular-value
decay in Theorem~\ref{thm:regular-decay}.
Proposition~\ref{prop:counting-transfer} compares counting functions
when the error operator is small enough. Section~\ref{sec:weyl}
applies this comparison to quadratic weights and computes the Weyl
constant. It also gives a weight with uniform Hessian bounds for
which no such constant exists.

The next corollary separates the local symbol exponent from both Fock
exponents. It follows from the preceding strong results and a local
$L^1$ estimate.
\begin{corollary}[An independent local exponent]
\label{cor:unified-descending}
Assume \eqref{eq:weight}. Let $1\le q<p<\infty$ and
$1\le\sigma<\infty$, put
\[
 s=\frac{pq}{p-q},
\]
and fix $r>0$. Suppose that
$f\in\Symb\cap\IDA^{s,\sigma}$. Then the following six assertions
are equivalent:
\begin{enumerate}
\item The initial Toeplitz operator $T_f$ on $\Gamma$ extends to a bounded
operator from $F^p_\varphi$ to $F^q_\varphi$.
\item The initial Toeplitz operator $T_f$ on $\Gamma$ extends to a compact
operator from $F^p_\varphi$ to $F^q_\varphi$.
\item The initial Toeplitz operator $T_f$ on $\Gamma$ extends to an
operator in $\Sch_s(F^2_\varphi)$.
\item The Berezin transform $b_f$ belongs to $L^s(\C^n,dv)$.
\item The local average $a_rf$ belongs to $L^s(\C^n,dv)$.
\item The local size function $M_{\sigma,r}f$ belongs to
$L^s(\C^n,dv)$.
\end{enumerate}
When these assertions hold, denote the extensions in \textup{(i)} and
\textup{(iii)} by $T_f^{p,q}$ and $T_f^{2,2}$, respectively. Set
$
 D:=\norm{G_{\sigma,r}f}_{L^s(dv)}.
$
Then
\begin{equation}\label{eq:unifiednorm}
 \begin{split}
 \bigl(\norm{T_f^{p,q}}_{p\to q}+D\bigr)
 &\asymp \bigl(\norm{T_f^{2,2}}_{\Sch_s}+D\bigr)
 \asymp \norm{M_{\sigma,r}f}_{L^s(dv)}\\
 &\asymp \bigl(\norm{b_f}_{L^s(dv)}+D\bigr)
 \asymp \bigl(\norm{a_rf}_{L^s(dv)}+D\bigr).
 \end{split}
\end{equation}
Thus $\sigma$, the exponent used in $G_{\sigma,r}f$ and
$M_{\sigma,r}f$, need not equal either Fock exponent $p$ or $q$.
The hypothesis $f\in\IDA^{s,\sigma}$ is still required.
\end{corollary}

Table~\ref{tab:results} summarizes the mapping and ideal criteria and
their symbol assumptions.

\begin{table}[!htbp]
\caption{Toeplitz criteria under the stated IDA assumptions.}
\label{tab:results}
\centering
\small
\begin{minipage}{\textwidth}
Assume \eqref{eq:weight}, let $f\in\Symb$, and fix $r>0$.
For $q<p$, put $s=pq/(p-q)$. All limits are as $|z|\to\infty$.
\end{minipage}\par
\medskip
\setlength{\tabcolsep}{4pt}
\renewcommand{\arraystretch}{1.15}
\begin{tabularx}{\textwidth}{@{}>{\raggedright\arraybackslash}p{.22\textwidth} >{\raggedright\arraybackslash}p{.20\textwidth} >{\raggedright\arraybackslash}p{.18\textwidth} >{\raggedright\arraybackslash}X@{}}
\toprule
Mapping and range & Symbol assumption & Property & Criterion or estimate \\
\midrule
$F^p_\varphi\to F^q_\varphi$\newline
$1\le p\le q<\infty$
& $f\in\BDA^q$
& Bounded\newline
Thm.~\ref{thm:mixed}(i)
& $b_f\in L^\infty$ \\
\addlinespace[6pt]
$F^p_\varphi\to F^q_\varphi$\newline
$1\le p\le q<\infty$
& $f\in\VDA^q$
& Compact\newline
Thm.~\ref{thm:mixed}(ii)
& $b_f(z)\to0$ \\
\addlinespace[6pt]
$F^p_\varphi\to F^q_\varphi$\newline
$1\le q<p<\infty$
& $f\in\IDA^{s,\sigma}$\newline
$1\le\sigma<\infty$
& Bounded\newline
$\Longleftrightarrow$ compact\newline
Cor.~\ref{cor:unified-descending}
& $b_f\in L^s$ \\
\addlinespace[6pt]
$F^p_\varphi\to F^q_\varphi$\newline
$1\le p\le q<\infty$
& $f\in\VDA^q$;\newline
$T_f$ bounded
& Essential norm\newline
Thm.~\ref{thm:mixed}(ii)
& $\begin{gathered}
\norm{T_f}_{\ess,p\to q}\asymp{}\\[-2pt]
\limsup_{|z|\to\infty}|b_f(z)|
\end{gathered}$ \\
\addlinespace[6pt]
$F^2_\varphi\to F^2_\varphi$\newline
$1\le q<\infty$, $\rho>0$\newline
$0<\tau\le\infty$
& $f\in\IDA^{(\rho,\tau),q}$
& $T_f\in\Sch_{\rho,\tau}$\newline
Thm.~\ref{thm:lorentz}
& $b_f\in L^{\rho,\tau}$ \\
\addlinespace[6pt]
\bottomrule
\end{tabularx}\par
\medskip
\begin{minipage}{\textwidth}
In the scalar rows, $b_f$ can be replaced by $a_rf$ or $M_{q,r}f$; the descending row uses $M_{\sigma,r}f$ instead.
That row is also equivalent to $T_f\in\Sch_s(F^2_\varphi)$.
In the essential norm row, these replacements give comparable quantities.
In the Schatten--Lorentz row, $q$ is the local symbol exponent and
$0<\tau\le\infty$.
The joint norm estimates are \eqref{eq:boundednorm},
\eqref{eq:descendingnorm}, \eqref{eq:lorentznorm}, and
\eqref{eq:unifiednorm}.
\end{minipage}
\end{table}

Section~\ref{sec:local} gives the local estimates.
Section~\ref{sec:mixedproof} proves the mixed mapping criteria. Section~\ref{sec:schattenproof} develops the full
Schatten--Lorentz theory, regular decay envelopes, and the counting
transfer. Section~\ref{sec:weyl} proves Weyl laws for quadratic weights
and identifies a sharp obstruction for general Hessian-bounded weights.
Section~\ref{sec:examples} gives sharpness and limitation examples.

\section{Preliminaries and local estimates}\label{sec:local}
All constants below may depend on the fixed geometric parameters and the displayed exponents. We write $A\lesssim B$ when $A\le CB$ with such a constant. We write $A\asymp B$ when both inequalities hold. An unlabelled $L^s$ norm is taken with respect to $dv$. Limits at infinity mean $|z|\to\infty$.

\subsection{Kernels, domains, and local averages}
The reproducing kernel and weighted submean estimates give the following
basic bounds. There exist constants \(c,C,c_0,r_0>0\) such that
\[
\begin{aligned}
K(z,z)e^{-2\varphi(z)}&\asymp 1,\\
|k_z(w)|e^{-\varphi(w)}&\le C e^{-c|z-w|},\\
|k_z(w)|e^{-\varphi(w)}&\ge c_0,
    \qquad |w-z|<r_0.
\end{aligned}
\]
Moreover, for every fixed \(k,r>0\),
\[
|g(z)|^k e^{-k\varphi(z)}
 \le C_{k,r}\int_{B(z,r)}
 |g(w)|^k e^{-k\varphi(w)}\,dv(w)
\]
whenever \(g\) is holomorphic on \(B(z,r)\).

The last estimate controls point evaluations by local weighted norms.
In particular,
$
 \|k_z\|_{k,\varphi}\asymp 1
$
for every fixed \(0<k<\infty\). See
\cite{Delin,SV,HVAPDE}. In local arguments we first work with small
radii and then use radius equivalence.

\FloatBarrier
The projection $P$ is bounded on $L^k_\varphi$ and reproduces
$F^k_\varphi$ for $1\le k<\infty$. The subspace $\Gamma$ is dense
in each such $F^k_\varphi$ \cite[Section 2]{HVAPDE}. By the weighted
submean estimate, norm convergence implies local uniform convergence.
The same estimate gives $F^p_\varphi\subset F^q_\varphi$ for $p\le q$.
Also, $F^q_\varphi$ is closed in $L^q_\varphi$. Hence the boundedness
and compactness of $T_f$ are unchanged if we regard it as a map
into $L^q_\varphi$.

The essential norm of a bounded $A:F^p_\varphi\to F^q_\varphi$ is
\[
 \norm{A}_{\ess,p\to q}
 =\inf\{\norm{A-L}_{p\to q}:L:F^p_\varphi\to F^q_\varphi\text{ is compact}\}.
\]

\begin{lemma}\label{lem:mean}
Let \(q\ge1\), \(0<t<\infty\), \(r>0\), and
\(f\in L^q_{\rm loc}(\mathbb C^n)\). Then
\begin{equation}\label{eq:meanrecovery}
 M_{q,r}f(z)\le C G_{q,4r}f(z)
    +C\left(\avint_{B(z,2r)}|a_rf(w)|^t\,dv(w)\right)^{1/t}.
\end{equation}
Consequently, for \(0<t\le\infty\),
\[
 \norm{M_{q,r}f}_t
 \asymp
 \norm{a_rf}_t+\norm{G_{q,r}f}_t.
\]
If \(f\in\VDA^q\), then
\[
 \limsup_{|z|\to\infty}M_{q,r}f(z)
 \asymp
 \limsup_{|z|\to\infty}|a_rf(z)|.
\]
\end{lemma}

\begin{proof}
Fix \(\varepsilon>0\), and choose
\(h\in\mathcal H(B(z,4r))\) such that
\[
 \left(\avint_{B(z,4r)}|f-h|^q\,dv\right)^{1/q}
 \le G_{q,4r}f(z)+\varepsilon.
\]
If \(w\in B(z,2r)\), then \(B(w,r)\subset B(z,4r)\). Hence the
mean-value property and H\"older's inequality give
\[
 |h(w)-a_rf(w)|
 \le C\bigl(G_{q,4r}f(z)+\varepsilon\bigr).
\]
The holomorphic submean inequality on \(B(z,2r)\) therefore yields
\[
 \sup_{B(z,r)}|h|
 \le C\bigl(G_{q,4r}f(z)+\varepsilon\bigr)
 +C\left(\avint_{B(z,2r)}|a_rf(w)|^t\,dv(w)\right)^{1/t}.
\]
Since \(q\ge1\), applying Minkowski's inequality to \(f=(f-h)+h\)
and letting \(\varepsilon\downarrow0\) proves
\eqref{eq:meanrecovery}.

For \(0<t<\infty\), integration of \eqref{eq:meanrecovery} and
Tonelli's theorem give
\[
 \norm{M_{q,r}f}_t
 \le C\bigl(\norm{G_{q,4r}f}_t+\norm{a_rf}_t\bigr).
\]
For \(t=\infty\), the same argument is used with the local supremum
of \(a_rf\). The radius equivalence
\cite[Corollary~3.10]{HVAPDE} replaces \(G_{q,4r}\) by \(G_{q,r}\).
The reverse estimates follow from
\[
 |a_rf|\le M_{q,r}f,
 \qquad
 G_{q,r}f\le M_{q,r}f.
\]
If \(f\in\VDA^q\), then \(G_{q,4r}f(z)\to0\), and the same pointwise
estimate gives the asserted limsup comparison.

Finally, if \(R>r\), a finite covering of \(B(0,R)\) by translates
of \(B(0,r)\) gives, for every \(t>0\),
\[
 M_{q,R}f(z)^t
 \le C_t\sum_j M_{q,r}f(z+v_j)^t.
\]
Thus the corresponding \(L^t\) norms, suprema, and limsups are
equivalent. This also applies when \(0<t<1\), since only a finite
sum is involved.
\end{proof}

\subsection{Hankel operators and local size estimates}
\begin{proposition}\label{prop:background}
Assume \eqref{eq:weight}, let \(1\le p,q<\infty\), and let
\(f\in\Symb\cap L^q_{\rm loc}\). Then:
\begin{enumerate}
\item
If \(1\le p\le q<\infty\), then
$
H_f:F^p_\varphi\to L^q_\varphi
$
is bounded if and only if \(f\in\BDA^q\), and compact if and only if
\(f\in\VDA^q\). Moreover,
\[
 \|H_f\|_{F^p_\varphi\to L^q_\varphi}
 \asymp \|G_{q,r}f\|_\infty .
\]

\item
If \(1\le q<p<\infty\) and
$
 s=\frac{pq}{p-q},
$
then \(H_f:F^p_\varphi\to L^q_\varphi\) is bounded if and only if
\(f\in\IDA^{s,q}\). In this case \(H_f\) is compact and
\[
 \|H_f\|_{F^p_\varphi\to L^q_\varphi}
 \asymp \|G_{q,r}f\|_s .
\]

\item
For \(p\le q\), the multiplication operator
$
M_f:F^p_\varphi\to L^q_\varphi
$
is bounded if and only if \(M_{q,r}f\in L^\infty\), and compact if
and only if
\[
M_{q,r}f(z)\to0
\qquad (|z|\to\infty).
\]
For \(p>q\), boundedness and compactness are equivalent to
$
 M_{q,r}f\in L^{pq/(p-q)}.
$
In each case the operator norm is comparable to the corresponding
local-size norm.
\end{enumerate}
\end{proposition}
\begin{proof}
The first two assertions are the Banach-range cases of
\cite[Theorem~1.1]{HVAPDE}. The initial symbol class used there
agrees with \eqref{eq:domain}, since normalized and unnormalized
reproducing kernels have the same finite linear span.

For part~(iii), set \(d\mu_f=|f|^q\,dv\). For \(g\in F^p_\varphi\),
\[
 \|M_fg\|_{q,\varphi}^q
 =\int_{\C^n}|g(w)|^qe^{-q\varphi(w)}\,d\mu_f(w).
\]
Multiplication by \(f\) is an isometry from
\(L^q(e^{-q\varphi}d\mu_f)\) onto a closed subspace of
\(L^q_\varphi\). Thus the Carleson embedding and the multiplication
operator have the same norm and compactness properties. The
\((p,q)\)-Carleson measure criterion applies;
see \cite{HL} and \cite[Section~2E]{HVAPDE}.
For every fixed \(r>0\),
\[
 \mu_f(B(z,r))
 =|B(0,r)|\,M_{q,r}f(z)^q.
\]
Thus, if \(p\le q\), boundedness is equivalent to
\(M_{q,r}f\in L^\infty\), while compactness is equivalent to
\(M_{q,r}f(z)\to0\) as \(|z|\to\infty\).
If \(p>q\), boundedness and compactness are both equivalent to
$
 M_{q,r}f\in L^{pq/(p-q)}.
$
The corresponding operator norms are comparable to the indicated
local-size norms.
\end{proof}
\begin{lemma}\label{lem:mixed-E}
Fix \(1\le q<\infty\). For
\(f\in\Symb\cap L^q_{\rm loc}\), put
$
 E_q(z)=\|H_fk_z\|_{q,\varphi}
$
with value \(+\infty\) when necessary. If \(G_{q,R}f\in L^\infty\), or
\(G_{q,R}f\in L^s\) for some \(s>q\), then
\[
 E_q(z)^q
 \le
 C\int_{\mathbb C^n}
 e^{-c|z-w|}G_{q,R}f(w)^q\,dv(w).
\]
Here \(R>0\) is fixed, and the constants may depend on \(R\).
Consequently,
\[
 \|E_q\|_\infty
 \le C\|G_{q,R}f\|_\infty,
\]
and, for \(s>q\),
\[
 \|E_q\|_s
 \le C\|G_{q,R}f\|_s.
\]
If \(f\in\VDA^q\), then
\[
 E_q(z)\longrightarrow0
 \qquad (|z|\to\infty).
\]
\end{lemma}

\begin{proof}
Put \(t=R/2\) and use the decomposition
\(f=f_1+f_2\) from \cite[Lemma~3.6]{HVAPDE}. Then
\(f_1\in C^2(\mathbb C^n)\) and
\begin{align}\label{X.1}
 |\bar\partial f_1(w)|
 +M_{q,t/2}f_2(w)
 \le C G_{q,R}f(w).
\end{align}
The same decomposition actually gives the corresponding local
\(L^q\) estimate for \(\bar\partial f_1\).

We first check the initial domains. By the exponential kernel estimate
and fixed-ball averaging,
\[
 \int_{\mathbb C^n}
 |f_2(w)k_z(w)|e^{-\varphi(w)}\,dv(w)
 \le
 C\int_{\mathbb C^n}
 e^{-c|z-u|}M_{q,t/2}f_2(u)\,dv(u).
\]
The right-hand side is finite if \(G_{q,R}f\) is bounded, and also if
\(G_{q,R}f\in L^s\) with \(s>q\), by H\"older's inequality.
Hence \(f_2\in\Symb\), and therefore \(f_1=f-f_2\in\Symb\).
The same argument with the \(q\)-th powers shows that
\[
 f_2k_z,\quad k_z\bar\partial f_1\in L_\varphi^q .
\]

By \cite[Lemma~2.4 and Corollary~2.5]{HVAPDE},
\[
 \|H_{f_1}k_z\|_{q,\varphi}
 \le C\|k_z\bar\partial f_1\|_{q,\varphi}.
\]
Since \(P\) is bounded on \(L_\varphi^q\),
\[
 \|H_{f_2}k_z\|_{q,\varphi}
 =
 \|(I-P)(f_2k_z)\|_{q,\varphi}
 \le C\|f_2k_z\|_{q,\varphi}.
\]
Thus the kernel decay gives
\[
 E_q(z)^q
 \le
 C\int_{\mathbb C^n}
 e^{-c|z-w|}
 \bigl(
  |\bar\partial f_1(w)|^q+|f_2(w)|^q
 \bigr)\,dv(w).
\]

For every fixed \(\tau>0\), Tonelli's theorem and the comparability
of \(e^{-c|z-w|}\) and \(e^{-c|z-u|}\) when \(|w-u|<\tau\) yield
\[
 \int_{\mathbb C^n}
 e^{-c|z-w|}|f_2(w)|^q\,dv(w)
 \le
 C\int_{\mathbb C^n}
 e^{-c|z-u|}M_{q,\tau}f_2(u)^q\,dv(u).
\]
Taking \(\tau=t/2\) and using \eqref{X.1}, together with
\(|\bar\partial f_1|\le C G_{q,R}f\), proves
\[
 E_q(z)^q
 \le
 C\int_{\mathbb C^n}
 e^{-c|z-w|}G_{q,R}f(w)^q\,dv(w).
\]

The \(L^\infty\) estimate is immediate. If \(s>q\), then
\(s/q>1\), and Young's inequality on \(L^{s/q}\) gives
$
 \|E_q\|_s^q
 \le
 C\|G_{q,R}f\|_s^q.
$
Finally, if \(f\in\VDA^q\), then radius equivalence gives
\(G_{q,R}f(z)\to0\). Since convolution with the integrable kernel
\(e^{-c|\cdot|}\) preserves vanishing at infinity, the preceding
pointwise estimate implies \(E_q(z)\to0\).
\end{proof}

\begin{lemma}
\label{lem:mixed-recovery}
Let \(1\le q<\infty\), \(f\in\Symb\cap L^q_{\rm loc}\), and let
\(E_q\) be as above. Write
$
 b_f(z)=\frac{T_fk_z(z)}{k_z(z)}.
$
There exists a sufficiently small fixed \(r>0\) such that
\[
 M_{q,r}f(z)
 \le C\bigl(E_q(z)+M_{q,r}b_f(z)+M_{q,r}E_q(z)\bigr).
\]
Changing the radii in the averaging terms by fixed factors, if needed,
does not affect any ensuing boundedness, vanishing, or \(L^s\) test.
In particular, for \(s>q\),
\[
 \|M_{q,r}f\|_s
 \le C\bigl(\|b_f\|_s+\|G_{q,R}f\|_s\bigr).
\]
The analogous supremum estimate holds. If \(f\in\VDA^q\), then
\[
 \limsup_{|z|\to\infty}M_{q,r}f(z)
 \le C\limsup_{|z|\to\infty}|b_f(z)|.
\]
\end{lemma}

\begin{proof}
Choose \(r>0\) so small that
\[
 |k_z(w)|e^{-\varphi(w)}\ge c_0,
 \qquad w\in B(z,8r).
\]
Then \(k_z\) has no zeros on \(B(z,8r)\), and
\[
 Q_z(w)=\frac{T_fk_z(w)}{k_z(w)},
 \qquad w\in B(z,8r),
\]
is holomorphic with \(Q_z(z)=b_f(z)\). Since
$
 (f-Q_z)k_z=H_fk_z,
$
the local kernel lower bound gives
$
 \|f-Q_z\|_{L^q(B(z,8r))}\le C E_q(z).
$

If \(w\in B(z,r)\), then
$
 B(w,2r)\subset B(z,8r)\cap B(w,8r).
$
The holomorphic submean inequality therefore gives
\[
 |Q_z(w)-b_f(w)|
 \le C\left(
  \fint_{B(w,2r)}|Q_z-Q_w|^q\,dv
 \right)^{1/q}
 \le C\bigl(E_q(z)+E_q(w)\bigr).
\]
Taking the local \(L^q\) mean over \(B(z,r)\), and using
Minkowski's inequality, yields
\[
 M_{q,r}f(z)
 \le C\bigl(E_q(z)+M_{q,r}b_f(z)+M_{q,r}E_q(z)\bigr).
\]

For \(s>q\), Young's inequality applied to the normalized average of
\(|h|^q\) gives
$
 \|M_{q,r}h\|_s\le C\|h\|_s.
$
Hence
\[
 \|M_{q,r}f\|_s
 \le C\bigl(\|b_f\|_s+\|E_q\|_s\bigr)
 \le C\bigl(\|b_f\|_s+\|G_{q,R}f\|_s\bigr),
\]
where the last inequality follows from
Lemma~\ref{lem:mixed-E}. The supremum estimate is identical.

Finally, if \(f\in\VDA^q\), then Lemma~\ref{lem:mixed-E} gives
\(E_q(z)\to0\). Since \(r\) is fixed, also
\(M_{q,r}E_q(z)\to0\). Thus
\[
 \limsup_{|z|\to\infty}M_{q,r}f(z)
 \le C\limsup_{|z|\to\infty}M_{q,r}b_f(z).
\]
Moreover,
\[
 M_{q,r}b_f(z)
 \le \sup_{w\in B(z,r)}|b_f(w)|,
\]
so
\[
 \limsup_{|z|\to\infty}M_{q,r}b_f(z)
 \le \limsup_{|z|\to\infty}|b_f(z)|.
\]
This proves the last assertion.
\end{proof}

\subsection{Berezin diagonals and compact evaluations}
\begin{lemma}
\label{lem:mixed-diagonal}
Let $1\le q<p<\infty$, put $s=pq/(p-q)$, and let
$A:F^p_\varphi\to F^q_\varphi$ be bounded. Define
\[
 d_A(z)=\frac{(Ak_z)(z)}{K(z,z)^{1/2}}.
\]
Then $d_A\in L^s(dv)$ and $\|d_A\|_s\le C\|A\|$.
\end{lemma}

\begin{proof}
Fix $\delta>0$. Let $\Lambda=\delta\mathbb Z^{2n}$, and let $Q$
be a fundamental cube. For $j\in\mathbb Z^{2n}$ and $v\in Q$, set
$\lambda_j=\delta j$ and $z_j=\lambda_j+v$. We first prove estimates
whose constants are independent of $v$.

For finitely supported $c$, set $S_vc=\sum_jc_jk_{z_j}$. Kernel decay gives
\[
 |S_vc(z)|e^{-\varphi(z)}\le C\sum_j|c_j|e^{-c_0|z-z_j|}.
\]
On the lattice cube associated with $z_i$, the right side is bounded by
$C\sum_j|c_j|e^{-c_1|i-j|}$. Since
$\{e^{-c_1|j|}\}\in\ell^1$, discrete Young's inequality and integration
over the cubes give
\[
 \norm{S_vc}_{p,\varphi}\le C\norm{c}_{\ell^p}.
\]
Thus $S_v:\ell^p\to F^p_\varphi$ is bounded uniformly in $v$.

Define $U_vg=\{g(z_j)/K(z_j,z_j)^{1/2}\}_j$. Using
$K(z,z)^{1/2}\asymp e^{\varphi(z)}$, the weighted submean estimate,
and bounded overlap, we obtain
\[
 \begin{split}
 \norm{U_vg}_{\ell^q}^q
 &\le C\sum_j\int_{B(z_j,\delta)}|g(z)|^qe^{-q\varphi(z)}\,dv(z)\\
 &\le C\norm{g}_{q,\varphi}^q.
 \end{split}
\]
Hence $U_v:F^q_\varphi\to\ell^q$ is also bounded uniformly in $v$.

For a finite index set $J$, compress $U_vAS_v$ to a matrix
$B_J:\ell^p(J)\to\ell^q(J)$. Its norm is at most $C\norm{A}$, and
its diagonal entries are $d_j=d_A(z_j)$. Let $D_\varepsilon$ be
multiplication by independent random signs. Then
\[
 \operatorname{diag}(B_J)=\mathbb E_\varepsilon
 D_\varepsilon B_JD_\varepsilon.
\]
Since the sign maps are isometries,
\[
 \norm{\operatorname{diag}(B_J)}_{\ell^p\to\ell^q}
 \le C\norm{A}.
\]
The norm of the diagonal map $c\mapsto\{d_jc_j\}$ is
\begin{equation}\label{eq:diagmultiplier}
 \norm{\operatorname{diag}(B_J)}_{\ell^p\to\ell^q}
    =\left(\sum_{j\in J}|d_j|^s\right)^{1/s}.
\end{equation}
Indeed, $1/q=1/p+1/s$ gives the upper estimate by H\"older's inequality.
For the reverse estimate, take $c_j=|d_j|^{s/p}$. Since
$q+sq/p=s$, the quotient of the two norms is the right side of
\eqref{eq:diagmultiplier}. The zero case is immediate.

Letting $J$ increase to $\mathbb Z^{2n}$ and using monotone convergence,
\[
 \sum_j|d_A(\lambda_j+v)|^s\le C\norm{A}^s
\]
uniformly in $v\in Q$. Therefore
\[
 \int_{\C^n}|d_A(z)|^s\,dv(z)
 =\int_Q\sum_j|d_A(\lambda_j+v)|^s\,dv(v)
 \le C|Q|\norm{A}^s.
\]
Finally, $d_A$ is measurable; in fact, $z\mapsto k_z$ is continuous
in $F^p_\varphi$, and normalized point evaluations vary continuously.
Thus $\|d_A\|_s\le C\|A\|$.
\end{proof}

\begin{lemma}
\label{lem:mixed-compactdiag}
For $1\le p,q<\infty$, a bounded map $A:F^p_\varphi\to F^q_\varphi$
satisfies $\|d_A\|_\infty\le C\|A\|$. If $A$ is compact, then
$d_A(z)\to0$. Consequently,
\[
 \limsup_{|z|\to\infty}|d_A(z)|\le C\|A\|_{\mathrm{ess}}.
\]
\end{lemma}

\begin{proof}
The kernels have uniformly bounded $F^p_\varphi$ norms, and the
functionals $\ell_z(g)=g(z)/K(z,z)^{1/2}$ have uniformly bounded
norms on $F^q_\varphi$. Hence
\[
 |d_A(z)|=|\ell_z(Ak_z)|\le C\|A\|.
\]
For each fixed $g\in F^q_\varphi$, the diagonal estimate
$K(z,z)^{1/2}\asymp e^{\varphi(z)}$ and the weighted submean inequality give
\[
 |\ell_z(g)|^q
 \le C\int_{B(z,r)}|g(w)|^qe^{-q\varphi(w)}\,dv(w)\to0
\]
as $|z|\to\infty$.

If $C\subset F^q_\varphi$ is compact, choose a finite $\epsilon$-net
$g_1,\ldots,g_N$ in $C$. Since $\|\ell_z\|\le C_0$,
\[
 \sup_{g\in C}|\ell_z(g)|
 \le \max_{1\le j\le N}|\ell_z(g_j)|+C_0\epsilon.
\]
Thus the convergence is uniform on compact subsets of $F^q_\varphi$.
If $A$ is compact, the closure of $\{Ak_z:z\in\mathbb C^n\}$ is
compact, and hence $d_A(z)\to0$.

Finally, for any compact $L:F^p_\varphi\to F^q_\varphi$,
$d_A=d_{A-L}+d_L$ and $d_L(z)\to0$. Therefore
\[
 \limsup_{|z|\to\infty}|d_A(z)|
 \le C\|A-L\|.
\]
Taking the infimum over compact $L$ gives the essential norm estimate.
This argument also covers $p=1$ and $q=1$ and does not use weak
convergence of the kernels.
\end{proof}

\section{Mixed mappings and compactness}\label{sec:mixedproof}
\begin{proof}[Proof of Theorem~\ref{thm:mixed}]
Suppose first that $p\le q$ and $f\in\BDA^q$. By
Proposition~\ref{prop:background}, $H_f:F^p_\varphi\to L^q_\varphi$
is bounded. On the common initial domain,
\begin{equation}\label{eq:column}
 M_f=JT_f+H_f,
\end{equation}
where $J:F^q_\varphi\hookrightarrow L^q_\varphi$ is the inclusion.
Thus boundedness of $T_f$ implies boundedness of $M_f$, while the
converse follows from $T_f=PM_f$.

To identify the bounded extension in \eqref{eq:column} with
multiplication, take $g_j\in\Gamma$ with $g_j\to g$ in
$F^p_\varphi$. Then $g_j\to g$ locally uniformly, while
$f g_j$ converges in $L^q_\varphi$. A subsequence converges almost
everywhere, and its limit is $fg$. Hence Proposition~\ref{prop:background}
gives
\[
 \norm{M_{q,r}f}_\infty
 \le C\bigl(\norm{T_f}_{p\to q}+\norm{G_{q,r}f}_\infty\bigr).
\]
Conversely, $T_f=PM_f$ and $G_{q,r}f\le M_{q,r}f$ give the reverse
estimate.

If $T_f$ is bounded, Lemma~\ref{lem:mixed-compactdiag} gives
$b_f\in L^\infty$. Conversely, boundedness of $b_f$ and
$G_{q,r}f$ gives boundedness of $M_{q,r}f$ by
Lemmas~\ref{lem:mixed-E} and \ref{lem:mixed-recovery}.
Lemma~\ref{lem:mean} gives the complex-average test. This proves
\eqref{eq:boundednorm}.

Next assume $f\in\VDA^q$. Then $H_f$ is compact. Hence
\eqref{eq:column} and Proposition~\ref{prop:background} show that
$T_f$ is compact if and only if $M_{q,r}f(z)\to0$.
If $T_f$ is compact, Lemma~\ref{lem:mixed-compactdiag} gives
$b_f(z)\to0$. Conversely, $b_f(z)\to0$ and
Lemma~\ref{lem:mixed-E} give $E_q(z)\to0$, so
Lemma~\ref{lem:mixed-recovery} yields $M_{q,r}f(z)\to0$.
Lemma~\ref{lem:mean} gives the complex-average criterion.

For the essential norm, let $f_N=f\mathbf1_{B(0,N)}$. The
multiplication operator associated with $f_N$ is compact from
$F^p_\varphi$ to $L^q_\varphi$: bounded sequences have locally
uniformly convergent subsequences, and $f\in L^q(B(0,N))$.
Thus $T_{f_N}$ is compact. Since
$M_{q,r}(f-f_N)(z)=0$ when $|z|\le N-r$,
Proposition~\ref{prop:background} gives
\[
 \norm{T_{f-f_N}}_{p\to q}
 \le C\sup_{|z|>N-r}M_{q,r}f(z).
\]
Letting $N\to\infty$ gives the essential-norm upper estimate by
$\limsup M_{q,r}f$. Lemma~\ref{lem:mixed-compactdiag} gives the
lower estimate by $\limsup|b_f|$, while under VDA,
Lemma~\ref{lem:mixed-recovery} gives the reverse comparison.
Finally, $|a_rf|\le M_{q,r}f$ and Lemma~\ref{lem:mean} complete
\eqref{eq:ess}.

Finally, let $q<p$ and $s=pq/(p-q)$. Then $s>q$. Under
$f\in\IDA^{s,q}$, Proposition~\ref{prop:background} shows that
$H_f$ is bounded and compact. Hence \eqref{eq:column} and the
multiplication criterion give
\[
 T_f\text{ bounded}
 \quad\Longleftrightarrow\quad
 M_{q,r}f\in L^s,
\]
and in this case $M_f$, and therefore $T_f$, is compact.
Conversely, compactness implies boundedness. The same identities,
together with $G_{q,r}f\le M_{q,r}f$, give the first norm comparison
in \eqref{eq:descendingnorm}.

If $T_f$ is bounded, Lemma~\ref{lem:mixed-diagonal} gives
$b_f\in L^s$ and
$\norm{b_f}_s\le C\norm{T_f}_{p\to q}$. Conversely,
$b_f,G_{q,r}f\in L^s$ imply $M_{q,r}f\in L^s$ by
Lemma~\ref{lem:mixed-recovery}; here Young's inequality is used with
$s/q>1$. Lemma~\ref{lem:mean} gives the complex-average equivalence
and the remaining norm comparisons.
\end{proof}

\section{Schatten--Lorentz ideals and the local symbol exponent}\label{sec:schattenproof}
We use the following standard singular-value facts; see \cite{Simon}.
For compact operators \(A,B\), bounded operators \(X,Y\), and
\(\alpha>0\),
\[
 s_{j+k-1}(A+B)\le s_j(A)+s_k(B),\qquad
 s_j(XAY)\le \|X\|\,\|Y\|\,s_j(A),
\]
and
$s_j(|A|^\alpha)=s_j(A)^\alpha.$
These give the Ky Fan estimates, the ideal property, and the
singular-value power identities used below. Inner products are linear
in the first variable.

We first prove the positive Toeplitz criterion. We then compare fixed
radii and construct a decomposition that preserves the initial domain.
The Hankel estimate and the local-size estimates complete the proof of
Theorem~\ref{thm:lorentz}. Its strong Schatten case follows by taking
$\tau=\rho$.

\subsection{Positive operators in Schatten--Lorentz ideals}
Positive local-mass estimates for radial Fock-type spaces appear in
\cite{HWZ}. Here we prove the fixed-radius version under
\eqref{eq:weight}. The proof uses the strong positive Toeplitz theorem
in \cite{IVW}, a threshold decomposition, and Hardy's inequality.
It also gives the counting estimates used below.

\begin{lemma}
\label{lem:positive-lorentz}
Assume \eqref{eq:weight}. Let $\mu$ be a positive locally finite Borel
measure. Put
\[
 \widetilde\mu(z)=\int_{\C^n}|k_z(w)|^2e^{-2\varphi(w)}\,d\mu(w).
\]
The notation $T_\mu$ refers to the bounded operator, when it exists,
represented by the positive form
\[
 \langle T_\mu u,v\rangle_\varphi
 =\int_{\C^n}u(w)\overline{v(w)}e^{-2\varphi(w)}\,d\mu(w).
\]
The Berezin integral is initially interpreted in $[0,\infty]$.
Existence of a bounded $T_\mu$ means that the displayed form is finite
for all $u,v\in F^2_\varphi$ and bounded on
$F^2_\varphi\times F^2_\varphi$. Local finiteness alone does not
ensure this property.

Fix $0<a<\infty$, $0<b\le\infty$, and $r>0$. Then
\[
 T_\mu\in\mathcal S_{a,b}
 \quad\Longleftrightarrow\quad
 \mu(B(\cdot,r))\in L^{a,b}
 \quad\Longleftrightarrow\quad
 \widetilde\mu\in L^{a,b},
\]
and the three quasi-norms are comparable.

More precisely, partition $\mathbb R^{2n}$ into half-open cubes $Q_j$
of a fixed sufficiently small side length, with centers $z_j$, and put
$m_j=\mu(Q_j)$. If $T_\mu$ is compact, then
\begin{equation}\label{eq:positivecountlower}
 \#\{j:m_j>t\}\le L N(ct;T_\mu),\qquad t>0.
\end{equation}
If $(m_j)\in\ell^{a,b}$ and $0<\eta<\min\{a,1\}$, then
\begin{equation}\label{eq:positivecountupper}
 N(Ct;T_\mu)\le C_\eta t^{-\eta}
 \sum_{m_j>t}m_j^\eta,
 \qquad t>0.
\end{equation}
Writing $(m_j^*)$ for the decreasing rearrangement of $(m_j)$, if
$b<\infty$ and $0<\eta<\min\{a,b,1\}$, then
\begin{equation}\label{eq:positive-rearrangement-upper}
 s_k(T_\mu)\le C_\eta
 \left(\frac1k\sum_{j=1}^k(m_j^*)^\eta\right)^{1/\eta},
 \qquad k\ge1.
\end{equation}
For $b=\infty$, the same estimate holds for
$0<\eta<\min\{a,1\}$.
Here $N(t;A)=\#\{k:s_k(A)>t\}$. The constants do not depend on
$\mu$ or $t$.
\end{lemma}

\begin{proof}
Choose the cubes so small that their diameter is less than $r$ and
the local kernel lower bound holds for any two points in one cube.
The kernel estimates give
\[
 |k_{z_i}(w)|e^{-\varphi(w)}
 \le C e^{-c_0|z_i-z_j|}
 \quad(w\in Q_j),\qquad
 |k_{z_i}(w)|e^{-\varphi(w)}
 \ge c_1\quad(w\in Q_i).
\]
The normalized kernels on any separated sublattice have a synthesis
operator $S:\ell^2\to F^2_\varphi$ with uniformly bounded norm.
These facts follow from \cite[Lemmas~2.1, 2.2 and 2.6]{IVW}.

Suppose $T_\mu$ is compact. Color the cube centers with $L$ colors so that centers of the same
color have pairwise distances at least $R$. The number $R$ will be
fixed below. Fix $t>0$ and a finite set $J$ of centers of one color
such that $m_j>t$. Define
\[
 d\nu=\sum_{j\in J}\frac{t}{m_j}\mathbf1_{Q_j}\,d\mu.
\]
Then $0\le\nu\le\mu$ and $\nu(Q_j)=t$ for $j\in J$. The finite matrix
$B=S^*T_\nu S$, where $Se_j=k_{z_j}$ for $j\in J$, satisfies
\[
 B_{ii}\ge c_1^2t,\qquad
 |B_{ij}|\le Ct\sum_{\ell\in J}
 e^{-c_0|z_i-z_\ell|}e^{-c_0|z_j-z_\ell|}.
\]
Put
\[
 \varepsilon_R=\sup_i\sum_{j\ne i}
 e^{-c_0|z_i-z_j|}.
\]
The supremum may be taken over the entire sublattice. A lattice
summation gives $\varepsilon_R\to0$ as $R\to\infty$. Separating the
term $\ell=i$ from the remaining terms gives
\[
 \sup_i\sum_{j\ne i}|B_{ij}|
 \le Ct\{\varepsilon_R+
 \varepsilon_R(1+\varepsilon_R)\}.
\]
Choose $R$ so that the right side is at most $c_1^2t/2$. The Schur
bound for the Hermitian off-diagonal matrix gives
\[
 B\ge \frac{c_1^2t}{2}I.
\]
Hence, for $x\in\ell^2(J)$,
\[
 \langle T_\mu Sx,Sx\rangle_\varphi
 \ge \langle T_\nu Sx,Sx\rangle_\varphi
 \ge \frac{c_1^2t}{2}\|x\|_2^2
 \ge \frac{c_1^2t}{2\|S\|^2}\|Sx\|_{2,\varphi}^2.
\]
The lower bound on $B$ makes $S$ injective. Its range has dimension
$|J|$, and $\|S\|$ is uniformly bounded. The min--max principle gives
$s_{|J|}(T_\mu)\ge c_2t$.
After decreasing the constant if necessary, the strict inequality in
the definition of $N$ gives
$|J|\le N(ct;T_\mu).$
Taking all finite subsets of each color and summing over the $L$
colors proves \eqref{eq:positivecountlower}.

For a nonnegative sequence $x=(x_j)$, put
$d_x(t)=\#\{j:x_j>t\}$.
For $b<\infty$, the distribution formula gives
\[
\|x\|_{\ell^{a,b}}^b
 \asymp\int_0^\infty
    [t\,d_x(t)^{1/a}]^b\,\frac{dt}{t},
\]
with the corresponding supremum formula when $b=\infty$.
Thus \eqref{eq:positivecountlower} implies
\begin{equation}\label{eq:positive-necessity-lorentz}
 \|(m_j)\|_{\ell^{a,b}}
 \le C\|T_\mu\|_{\mathcal S_{a,b}}.
\end{equation}

Assume $(m_j)\in\ell^{a,b}$. Its embedding into
$\ell^{a,\infty}$ gives uniformly bounded cube masses. Since each
fixed-radius ball meets only a bounded number of cubes, the positive
Carleson criterion defines a bounded $T_\mu$. The same bound and a
lattice summation give
\[
 \int_{\C^n}e^{-c|z-w|}\,d\mu(w)<\infty
 \qquad(c>0,\ z\in\C^n),
\]
which is the measure hypothesis used in \cite{IVW}.

For $t>0$, write
$\mu=\mu_{\le t}+\mu_{>t},$
where the two measures are obtained by restricting $\mu$ to cubes
with $m_j\le t$ and $m_j>t$, respectively. Since each fixed-radius
ball meets only finitely many cubes,
\[
 \sup_z\mu_{\le t}(B(z,r))\le Ct.
\]
Hence the positive Carleson estimate gives
$\|T_{\mu_{\le t}}\|\le C_0t.$

Let $0<\eta<\min\{a,1\}$. Since $\eta\le1$,
\[
 \left(\sum_\ell x_\ell\right)^\eta
 \le\sum_\ell x_\ell^\eta
 \qquad(x_\ell\ge0).
\]
Using again the bounded number of cubes meeting a fixed ball gives
\[
 \|\mu_{>t}(B(\cdot,r))\|_{L^\eta}^{\eta}
 \le C\sum_{m_j>t}m_j^\eta.
\]
The strong positive Toeplitz theorem
\cite[Theorem~2.7]{IVW} therefore yields
\begin{equation}\label{eq:positive-threshold-split}
 \|T_{\mu_{\le t}}\|\le C_0t,
 \qquad
 \|T_{\mu_{>t}}\|_{\mathcal S_\eta}^{\eta}
 \le C_\eta\sum_{m_j>t}m_j^\eta.
\end{equation}

The sum on the right is finite. Indeed, if
$W=\|(m_j)\|_{\ell^{a,\infty}},$
then $d_m(\lambda)\le W^a\lambda^{-a}$, and the layer-cake formula
gives, since $\eta<a$,
\[
 \sum_{m_j>t}m_j^\eta
 \le \frac{a}{a-\eta}W^a t^{\eta-a}.
\]
Thus $T_{\mu_{>t}}$ is compact, and
$T_\mu$ is its operator-norm limit as $t\downarrow0$. Moreover, the
singular-value perturbation inequality and
\eqref{eq:positive-threshold-split} give
\eqref{eq:positivecountupper}.

Let $x_j=m_j^*$. If $x_k=0$, then $x_j=0$ for $j\ge k$. The strong
positive theorem gives
\[
 s_k(T_\mu)
 \le k^{-1/\eta}\|T_\mu\|_{\mathcal S_\eta}
 \le Ck^{-1/\eta}
    \left(\sum_{j<k}x_j^\eta\right)^{1/\eta},
\]
so \eqref{eq:positive-rearrangement-upper} already follows.

Suppose now that $x_k>0$ and take $t=x_k$. At most $k-1$ cube masses
are strictly larger than $t$. By
\eqref{eq:positive-threshold-split},
\[
 \begin{split}
 s_k(T_\mu)
 &\le \|T_{\mu_{\le x_k}}\|
   +s_k(T_{\mu_{>x_k}})\\
 &\le Cx_k+Ck^{-1/\eta}
    \left(\sum_{j<k}x_j^\eta\right)^{1/\eta}\\
 &\le C\left(\frac1k\sum_{j\le k}x_j^\eta\right)^{1/\eta}.
 \end{split}
\]
This proves \eqref{eq:positive-rearrangement-upper}.

If $b<\infty$, choose
$0<\eta<\min\{a,b,1\}$
and put
\[
 p=\frac a\eta>1,\qquad s=\frac b\eta>1.
\]
The weighted discrete Hardy inequality
\cite[Chapter~2]{BennettSharpley},
\[
 \sum_{k\ge1}k^{s/p-1}
 \left(\frac1k\sum_{j\le k}y_j\right)^s
 \le C_{p,s}\sum_{k\ge1}k^{s/p-1}y_k^s,
\]
applied to $y_j=x_j^\eta$, gives
\begin{equation}\label{eq:positive-sufficiency-lorentz}
 \|T_\mu\|_{\mathcal S_{a,b}}
 \le C\|(m_j)\|_{\ell^{a,b}}.
\end{equation}
If $b=\infty$, choose $0<\eta<\min\{a,1\}$ and use
\[
 k^{1/a}
 \left(\frac1k\sum_{j\le k}x_j^\eta\right)^{1/\eta}
 \le C_{a,\eta}\|x\|_{\ell^{a,\infty}}.
\]
This proves \eqref{eq:positive-sufficiency-lorentz} also for
$b=\infty$. In particular, the argument covers $a<1$ and $b<1$.

For $z\in Q_j$, the choice of the cube size gives
$
 Q_j\subset B(z,r),
$
and hence
$
 \mu(B(z,r))\ge m_j.
$

Conversely, each ball $B(z,r)$ meets at most a fixed number of cubes.
Solidity and the finite-sum quasi-triangle inequality in Lorentz
spaces therefore give
\[
 \|(m_j)\|_{\ell^{a,b}}
 \asymp
 \|\mu(B(\cdot,r))\|_{L^{a,b}}.
\]

For $z\in Q_j$, the kernel bounds give
\[
 cm_j\le\widetilde\mu(z)
 \le C\sum_i m_i e^{-c|z_i-z_j|}.
\]
Choose
$
 0<\delta<\min\{1,a,b\}
$
when $b<\infty$, and
$
 0<\delta<\min\{1,a\}
$
when $b=\infty$. Since $\delta\le1$,
\[
 \left(\sum_i m_i e^{-c|z_i-z_j|}\right)^\delta
 \le C\sum_i m_i^\delta
    e^{-c\delta|z_i-z_j|}.
\]
Indexing the lattice cubes by $\mathbb Z^{2n}$, convolution by this
summable sequence is bounded with respect to an equivalent Banach
lattice norm on $\ell^{a/\delta,b/\delta}$, also when $b=\infty$.
Comparing $\widetilde\mu$ with the corresponding step function and
using the power identity for Lorentz quasi-norms gives
\[
 \|\widetilde\mu\|_{L^{a,b}}
 \asymp\|(m_j)\|_{\ell^{a,b}}.
\]
Together with \eqref{eq:positive-necessity-lorentz} and
\eqref{eq:positive-sufficiency-lorentz}, this completes the proof.
\end{proof}

\subsection{Local decomposition and Hankel estimates}

\begin{lemma}
\label{lem:lorentz-radius}
Let $1\le q<\infty$, $0<\rho<\infty$, $0<\tau\le\infty$,
$r,R>0$, and $f\in L^q_{\rm loc}(\C^n)$. Then
\begin{equation}\label{eq:lorentz-radius}
 \|G_{q,r}f\|_{\rho,\tau}\asymp\|G_{q,R}f\|_{\rho,\tau}.
\end{equation}
If these quantities are finite, then $G_{q,R}f$ is bounded and
$G_{q,R}f(z)\to0$ as $|z|\to\infty$, for every fixed $R>0$.
Consequently, $\IDA^{(\rho,\tau),q}\subset\VDA^q\subset\BDA^q$.
\end{lemma}

\begin{proof}
Choose $s>0$ with $2s<r$. By \cite[Lemma~3.6]{HVAPDE}, write
$f=f_1+f_2$, where $f_1\in C^2$ and
\begin{equation}\label{eq:lorentz-ida-decomp}
 M_{q,s/2}(\bar\partial f_1)+M_{q,s/2}f_2
 \le C G_{q,2s}f\le C_rG_{q,r}f.
\end{equation}
Here $M_{q,a}(\bar\partial f_1)$ is the local $L^q$ mean of the
Euclidean norm of $\bar\partial f_1$. The local solution estimates in
\cite[Theorem~3.8, equations~(3.22)--(3.24)]{HVAPDE} give
\[
 G_{q,R}f_1\le C_RM_{q,2R}(\bar\partial f_1),\qquad
 G_{q,R}f_2\le M_{q,R}f_2.
\]
Finite coverings bound each mean on the right by a finite sum of
translates of the corresponding mean at radius $s/2$. Together with
\eqref{eq:lorentz-ida-decomp}, this gives a pointwise bound of
$G_{q,R}f$ by a finite sum of translates of $G_{q,r}f$.
Translation invariance and the Lorentz quasi-triangle inequality give
$\|G_{q,R}f\|_{\rho,\tau}\le C_{r,R}\|G_{q,r}f\|_{\rho,\tau}.$
Interchanging $r$ and $R$ proves \eqref{eq:lorentz-radius}.

Restriction of holomorphic approximants gives
\begin{equation}\label{eq:lorentz-local-propagation}
 G_{q,R/2}f(z)\le C_RG_{q,R}f(w),\qquad w\in B(z,R/2);
\end{equation}
see \cite[Corollary~3.4]{HVAPDE}. If $A=G_{q,R/2}f(z)>0$, then
$B(z,R/2)\subset\{G_{q,R}f>A/(2C_R)\}$. The weak distribution bound
and $L^{\rho,\tau}\hookrightarrow L^{\rho,\infty}$ imply
$\|G_{q,R/2}f\|_\infty\le C\|G_{q,R}f\|_{\rho,\tau}.$
If $G_{q,R/2}f$ did not tend to zero, there would be
$\varepsilon>0$ and $z_j\to\infty$ with
$G_{q,R/2}f(z_j)\ge\varepsilon$. Select a subsequence for which
$B(z_j,R/2)$ are disjoint. By \eqref{eq:lorentz-local-propagation},
the set $\{G_{q,R}f>\varepsilon/(2C_R)\}$ would have infinite
measure, contrary to $G_{q,R}f\in L^{\rho,\infty}$.
Applying the argument with $2R$ in place of $R$ proves both conclusions
at every fixed radius. The stated inclusions follow from the definitions.
\end{proof}

\begin{lemma}[Reduction to a local square-integrable symbol]
\label{lem:q-decomposition}
Assume \eqref{eq:weight}. Let $1\le q<\infty$, $t>0$, and
$f\in\Symb\cap\BDA^q$. Put $\Delta=G_{q,16t}f$.
There is a decomposition $f=g+h$ with $g\in C^\infty(\C^n)$ such that
\[
 M_{q,t/2}h\le C G_{q,2t}f,\qquad G_{2,t}g\le C\Delta.
\]
Both $g$ and $h$ belong to $\Symb$. Also, $gk_\zeta\in L^2_\varphi$
for every $\zeta\in\C^n$. The local sizes satisfy
\[
 \begin{split}
 M_{q,t}f&\le M_{q,t}h+C M_{2,2t}g+C\Delta,\\
 M_{2,t}g&\le C M_{q,2t}f+C\Delta.
 \end{split}
\]
All constants are independent of $f$.
\end{lemma}

\begin{proof}
Choose a lattice $\{a_j\}$, a nonnegative smooth partition of
unity $\{\psi_j\}$ supported in $B(a_j,t/2)$, and local holomorphic best
approximants $h_j$ on $B(a_j,t)$.
Lemma~3.3 of \cite{HVAPDE} shows that these local best approximants
exist for every $1\le q<\infty$.
Set
\[
 g=\sum_j\psi_jh_j,\qquad h=f-g.
\]
By \cite[Lemmas 3.3 and 3.6]{HVAPDE},
\begin{equation}\label{eq:q-decomp-error}
 M_{q,t/2}h(z)\le C G_{q,2t}f(z).
\end{equation}
We also need a local consequence of the same construction. Let $v_z$
be a best holomorphic approximant to $f$ on $B(z,16t)$. Then
\begin{equation}\label{eq:q-decomp-sup}
 \sup_{w\in B(z,4t)}|g(w)-v_z(w)|
 \le C G_{q,16t}f(z).
\end{equation}
Indeed, if $w\in B(z,4t)$ and $\psi_j(w)\ne0$, then
$B(w,t/4)\subset B(a_j,t)\subset B(z,16t)$. The best-approximation
property gives
\[
 M_{q,t}(f-h_j)(a_j)
 \le M_{q,t}(f-v_z)(a_j)\le C G_{q,16t}f(z).
\]
The same bound holds for $M_{q,t}(h_j-v_z)(a_j)$ by the triangle
inequality. Apply the holomorphic submean inequality on $B(w,t/4)$.
It follows that $|h_j(w)-v_z(w)|\le C G_{q,16t}f(z)$.
Summing against $\psi_j(w)$ proves \eqref{eq:q-decomp-sup}. Hence
\begin{equation}\label{eq:q-to-two}
 G_{2,t}g(z)\le C G_{q,16t}f(z).
\end{equation}

Since $q\ge1$, the local $L^1$ means of $h$ are bounded by its local
$L^q$ means. These are bounded by \eqref{eq:q-decomp-error}.
Kernel decay and a lattice summation therefore give
$hk_\zeta\in L^1_\varphi$ for every $\zeta$, and hence
$gk_\zeta\in L^1_\varphi$. In fact $gk_\zeta\in L^2_\varphi$.
Using the weighted holomorphic $L^1$ submean estimate for
$v_zk_\zeta$ and \eqref{eq:q-decomp-sup}, with
$D_\infty=\|G_{q,16t}f\|_\infty$, gives
\[
 |g(z)k_\zeta(z)|e^{-\varphi(z)}
 \le C\avint_{B(z,t)}|g(w)k_\zeta(w)|e^{-\varphi(w)}\,dv(w)
 +CD_\infty\sup_{w\in B(z,t)}
 |k_\zeta(w)|e^{-\varphi(w)}.
\]
The right side is uniformly bounded in $z$. Thus
$gk_\zeta e^{-\varphi}\in L^1\cap L^\infty\subset L^2$, so $g$
satisfies the claimed kernel $L^2$ condition.
Moreover, $T_f=T_g+T_h$ on $\Gamma$.

Set $\Delta(z)=G_{q,16t}f(z)$. From
\eqref{eq:q-decomp-sup} and the holomorphic submean inequality,
\[
 \sup_{B(z,t)}|v_z|
 \le C M_{2,2t}v_z(z)
 \le C M_{2,2t}g(z)+C\Delta(z).
\]
Since $f=g+h$,
\begin{equation}\label{eq:sizeforward}
 M_{q,t}f(z)
 \le M_{q,t}h(z)+C M_{2,2t}g(z)+C\Delta(z).
\end{equation}
Conversely,
\[
 \sup_{B(z,t)}|v_z|
 \le C M_{q,2t}v_z(z)
 \le C M_{q,2t}f(z)+C\Delta(z),
\]
and therefore
\begin{equation}\label{eq:sizereverse}
 M_{2,t}g(z)
 \le C M_{q,2t}f(z)+C\Delta(z).
\end{equation}

These are the two local-size estimates in the statement.
\end{proof}

\begin{lemma}
\label{lem:lorentz-hankel}
Assume \eqref{eq:weight}. Let $0<\rho<\infty$,
$0<\tau\le\infty$, and let $u\in\Symb\cap L^2_{\rm loc}$ satisfy
$uk_z\in L^2_\varphi$ for every $z$. If
$G_{2,r}u\in L^{\rho,\tau}$, then
\[
 H_u\in\Sch_{\rho,\tau},\qquad
 E_u(z):=\|H_uk_z\|_{2,\varphi}\in L^{\rho,\tau},
\]
and
\begin{equation}\label{eq:lorentz-hankel-estimate}
 \|H_u\|_{\Sch_{\rho,\tau}}+\|E_u\|_{\rho,\tau}
 \le C\|G_{2,r}u\|_{\rho,\tau}.
\end{equation}
\end{lemma}

\begin{proof}
We combine the IDA decomposition with a positive-form estimate. A
related argument for radial Fock-type spaces appears in \cite{HWZ}.
Choose $t>0$ with
$2t<r$. By \cite[Lemma~3.6]{HVAPDE}, with local exponent $2$, write
$u=u_1+u_2$, where $u_1\in C^2$ and
\begin{equation}\label{eq:lorentz-hankel-decomp}
 M_{2,t/2}(\bar\partial u_1)+M_{2,t/2}u_2
 \le C G_{2,2t}u
 \le C_rG_{2,r}u.
\end{equation}
The local $L^2$ mean of $\bar\partial u_1$ is the quantity appearing
in the corrected estimate of
\cite[Theorem~2.6 and Remark~2]{HVCorr}.

By Lemma~\ref{lem:lorentz-radius}, the right side of
\eqref{eq:lorentz-hankel-decomp} is bounded. Hence the corresponding
fixed-ball $L^2$ masses are uniformly bounded. Kernel decay gives
$u_2k_z,\ k_z\bar\partial u_1\in L^2_\varphi.$
Since $uk_z\in L^2_\varphi$, we also have $u_1k_z\in L^2_\varphi$.
Thus the identities below are justified first on $\Gamma$.

Set
\[
 d\mu_1=|\bar\partial u_1|^2\,dv,\qquad
 d\mu_2=|u_2|^2\,dv.
\]
From \eqref{eq:lorentz-hankel-decomp},
\[
 \|\mu_j(B(\cdot,t/2))\|_{\rho/2,\tau/2}
 \le C\|G_{2,r}u\|_{\rho,\tau}^2,
 \qquad j=1,2,
\]
where $\infty/2=\infty$. Lemma~\ref{lem:positive-lorentz} therefore
gives
\[
 T_{\mu_j}\in\Sch_{\rho/2,\tau/2},
 \qquad
 \|T_{\mu_j}\|_{\Sch_{\rho/2,\tau/2}}
 \le C\|G_{2,r}u\|_{\rho,\tau}^2.
\]

For $v\in\Gamma$, the weighted $\bar\partial$ estimate
\cite[Section~4]{HVAPDE} and the contractivity of the orthogonal
projection give
\[
 \|H_{u_1}v\|_{2,\varphi}^2
 \le C\int_{\C^n}|v|^2e^{-2\varphi}\,d\mu_1,
 \qquad
 \|H_{u_2}v\|_{2,\varphi}^2
 \le \int_{\C^n}|v|^2e^{-2\varphi}\,d\mu_2.
\]
Here $H_{u_1}v$ is the minimal weighted $L^2$ solution of
$\bar\partial w=v\bar\partial u_1$. Since $T_{\mu_j}$ are bounded,
these estimates extend $H_{u_j}$ from $\Gamma$ to $F^2_\varphi$ and give
\[
 0\le H_{u_1}^*H_{u_1}\le CT_{\mu_1},
 \qquad
 0\le H_{u_2}^*H_{u_2}\le T_{\mu_2}.
\]

Fix $j\in\{1,2\}$. Let $E_N$ project onto the eigenvectors for the
$N$ largest positive eigenvalues of $T_{\mu_j}$, or onto its whole
range if its rank is less than $N$. Then
\[
 \|H_{u_j}(I-E_N)\|^2
 \le C s_{N+1}(T_{\mu_j})\longrightarrow0,
\]
while $H_{u_j}E_N$ has finite rank. The min--max principle gives
$s_k(H_{u_j})^2\le Cs_k(T_{\mu_j}),$
and therefore
\[
 \|H_{u_j}\|_{\Sch_{\rho,\tau}}^2
 \le C\|T_{\mu_j}\|_{\Sch_{\rho/2,\tau/2}}
 \le C\|G_{2,r}u\|_{\rho,\tau}^2.
\]
Since $H_u=H_{u_1}+H_{u_2}$,
\[
 \|H_u\|_{\Sch_{\rho,\tau}}
 \le C\|G_{2,r}u\|_{\rho,\tau}.
\]

Finally, testing the same form inequalities on $k_z$ yields
\[
 \|H_{u_j}k_z\|_{2,\varphi}^2
 \le C\widetilde\mu_j(z).
\]
By Lemma~\ref{lem:positive-lorentz},
\[
 \|\widetilde\mu_j\|_{\rho/2,\tau/2}
 \le C\|G_{2,r}u\|_{\rho,\tau}^2.
\]
Using
\[
 \|\widetilde\mu_j^{1/2}\|_{\rho,\tau}^2
 =\|\widetilde\mu_j\|_{\rho/2,\tau/2}
\]
and
$E_u\le E_{u_1}+E_{u_2},$
the Lorentz quasi-triangle inequality gives
\[
 \|E_u\|_{\rho,\tau}
 \le C\|G_{2,r}u\|_{\rho,\tau}.
\]
Together with the preceding operator estimate, this proves
\eqref{eq:lorentz-hankel-estimate}.
\end{proof}

\subsection{Proof of the Schatten--Lorentz criterion}

\begin{proof}[Proof of Theorem~\ref{thm:lorentz}]
We first establish two scalar estimates valid at all Lorentz indices.
For these estimates let $v\in\Symb\cap L^q_{\rm loc}$.
Choose
\[
 0<\theta<\min\{1,\rho,\tau\}
 \quad\text{if }\tau<\infty,
 \qquad
 0<\theta<\min\{1,\rho\}
 \quad\text{if }\tau=\infty.
\]
Then $\rho/\theta>1$ and $\tau/\theta>1$ when $\tau<\infty$.
We use an equivalent Banach lattice norm on
$L^{\rho/\theta,\tau/\theta}$. For $p>1$ and
$1\le s\le\infty$, this norm may be defined by replacing $w^*$ with
\[
 w^{**}(t)=\frac1t\int_0^t w^*(u)\,du;
\]
see \cite[Chapter~4]{BennettSharpley}. For every nonnegative
$\kappa\in L^1(\C^n)$,
$(\kappa*|w|)^{**}(t)\le\|\kappa\|_1w^{**}(t).$
Hence fixed-ball averaging and integrable convolution are bounded in
the convexified Lorentz space. Together with
\[
 \||v|^\theta\|_{\rho/\theta,\tau/\theta}
 =\|v\|_{\rho,\tau}^\theta,
\]
\eqref{eq:meanrecovery} with exponent $\theta$ and
Lemma~\ref{lem:lorentz-radius} give
\begin{equation}\label{eq:lorentz-mean-recovery}
 \|M_{q,r}v\|_{\rho,\tau}
 \le C\bigl(\|a_rv\|_{\rho,\tau}
       +\|G_{q,r}v\|_{\rho,\tau}\bigr).
\end{equation}
Likewise, choose a lattice spacing $\delta>0$ so small that each cube
centered at $z+\delta k$ lies in $B(z+\delta k,r)$. Kernel decay gives
\[
 |b_v(z)|\le C\sum_{k\in\mathbb Z^{2n}}e^{-c|k|}
          M_{q,r}v(z+\delta k).
\]
Raising to the power $\theta$ and using translation invariance and
the preceding convolution estimate yields
\begin{equation}\label{eq:lorentz-b-upper}
 \|b_v\|_{\rho,\tau}\le C\|M_{q,r}v\|_{\rho,\tau}.
\end{equation}

We next treat the local $L^2$ case. Let $v\in\Symb$ satisfy
$vk_z\in L^2_\varphi$ for every $z$ and
$G_{2,r}v\in L^{\rho,\tau}$. Lemma~\ref{lem:lorentz-hankel} gives
\begin{equation}\label{eq:lorentz-local-hankel}
 H_v\in\Sch_{\rho,\tau},\qquad
 E_v(z)=\|H_vk_z\|_{2,\varphi}\in L^{\rho,\tau},
\end{equation}
with both quasi-norms bounded by
$C\|G_{2,r}v\|_{\rho,\tau}$.

The identity $M_v=JT_v+H_v$ holds first on $\Gamma$. Since
$J:F^2_\varphi\hookrightarrow L^2_\varphi$ is isometric, an extension
of either $T_v$ or $M_v$ in $\Sch_{\rho,\tau}$ yields the corresponding
extension of the other.

We next check that a bounded extension of $M_v$ acts by multiplication.
If $M_{2,r}v\in L^{\rho,\tau}$, then for
$w\in B(z,r)$,
$M_{2,r}v(z)\le C M_{2,2r}v(w).$
Finite coverings give
\[
 \|M_{2,2r}v\|_{\rho,\tau}
 \le C\|M_{2,r}v\|_{\rho,\tau}.
\]
If $M_{2,r}v(z)=A>0$, the propagation estimate puts $B(z,r)$ inside
$\{M_{2,2r}v>A/(2C)\}$. Its fixed positive volume and the weak
$L^\rho$ bound imply $A\le C\|M_{2,2r}v\|_{\rho,\infty}$.
Thus $M_{2,r}v$ is bounded. The positive Carleson criterion therefore
defines the multiplication operator
$M_v:F^2_\varphi\to L^2_\varphi$.

If a bounded extension of $M_v$ is given instead, choose
$u_j\in\Gamma$ with $u_j\to u$ in $F^2_\varphi$. Then
$u_j\to u$ locally uniformly, and a subsequence of $vu_j$ converges
almost everywhere to the $L^2_\varphi$ limit. Hence that limit is
$vu$. Thus in either case the bounded extension is multiplication, and
$M_v^*M_v=T_{|v|^2}$
as a positive form. The singular-value power identity and
Lemma~\ref{lem:positive-lorentz} now give
\[
\begin{split}
 T_v\in\Sch_{\rho,\tau}
 &\Longleftrightarrow M_v\in\Sch_{\rho,\tau}
 \Longleftrightarrow T_{|v|^2}\in\Sch_{\rho/2,\tau/2}\\
 &\Longleftrightarrow M_{2,r}v\in L^{\rho,\tau}.
 \end{split}
\]

Fix $r_0>0$ so small that the uniform local kernel lower bound holds on
$B(z,8r_0)$ for every $z$. At this radius use the holomorphic quotients
\[
 Q_z=\frac{T_vk_z}{k_z}.
\]
The quotient comparison in Lemma~\ref{lem:mixed-recovery} gives
$|Q_z(w)-b_v(w)|\le C(E_v(z)+E_v(w))$ for $w\in B(z,2r_0)$.
Also, $\|v-Q_z\|_{L^2(B(z,4r_0))}\le CE_v(z)$.
Apply the holomorphic submean inequality to $Q_z$ with exponent
$\theta$. We obtain
\[
 M_{2,r_0}v(z)\le CE_v(z)+C\left(\avint_{B(z,2r_0)}
    (|b_v(w)|^\theta+E_v(w)^\theta)\,dv(w)\right)^{1/\theta}.
\]
The convexified Lorentz averaging estimate and
\eqref{eq:lorentz-local-hankel} give
\[
 \|M_{2,r_0}v\|_{\rho,\tau}
 \le C\bigl(\|b_v\|_{\rho,\tau}
       +\|G_{2,r}v\|_{\rho,\tau}\bigr).
\]
Finite coverings compare the fixed-radius local means, while
Lemma~\ref{lem:lorentz-radius} compares the analytic distances. Hence
\begin{equation}\label{eq:lorentz-local-recovery}
 \|M_{2,r}v\|_{\rho,\tau}
 \le C\bigl(\|b_v\|_{\rho,\tau}
       +\|G_{2,r}v\|_{\rho,\tau}\bigr).
\end{equation}
Equations \eqref{eq:lorentz-mean-recovery},
\eqref{eq:lorentz-b-upper}, and
\eqref{eq:lorentz-local-recovery} prove the scalar equivalences and
the joint estimates in the local $L^2$ case.

For general $q\ge1$, fix a small $t>0$ and use the decomposition $f=g+h$ in
Lemma~\ref{lem:q-decomposition}. Its BDA assumption follows from
Lemma~\ref{lem:lorentz-radius}. Put $\Delta=G_{q,16t}f$.
Equations \eqref{eq:q-decomp-error} and
\eqref{eq:q-to-two}, together with Lemma~\ref{lem:lorentz-radius},
give
\begin{equation}\label{eq:lorentz-general-decomp}
 \|G_{2,t}g\|_{\rho,\tau}
 +\|M_{q,t/2}h\|_{\rho,\tau}
 \le CD_{\rho,\tau}.
\end{equation}
The boundedness statement in Lemma~\ref{lem:lorentz-radius} also gives
the domain argument following \eqref{eq:q-decomp-sup}; hence
$g,h\in\Symb$ and $gk_z\in L^2_\varphi$.

Put $d\mu=|h|\,dv$. Since
$M_{1,t/2}h\le M_{q,t/2}h,$
the fixed-ball masses of $\mu$ belong to $L^{\rho,\tau}$ with norm
bounded by $CD_{\rho,\tau}$. Lemma~\ref{lem:positive-lorentz} therefore
gives
$T_\mu\in\Sch_{\rho,\tau}.$
Let $A:F^2_\varphi\to L^2_\varphi$ be multiplication by
$|h|^{1/2}$. Then $A^*A=T_\mu$. With
$\operatorname{sgn}h=h/|h|$ on $\{h\ne0\}$ and zero elsewhere,
$T_h=A^*M_{\operatorname{sgn}h}A$
in the sense of forms, and this factorization agrees with the initial
Toeplitz operator on $\Gamma$. The singular-value inequality
\[
 s_{2k-1}(T_h)
 \le s_k(A)^2
 =s_k(T_\mu)
\]
gives the operator estimate. The scalar estimates follow from
\[
 |a_rh|\le M_{1,r}h,\qquad
 |b_h|\le b_{|h|},
\]
together with fixed-radius comparison and
Lemma~\ref{lem:positive-lorentz}. Thus
\begin{equation}\label{eq:lorentz-remainder}
 \|T_h\|_{\Sch_{\rho,\tau}}
 +\|a_rh\|_{\rho,\tau}
 +\|b_h\|_{\rho,\tau}
 \le CD_{\rho,\tau}.
\end{equation}

Apply the local $L^2$ result to $g$. Equations
\eqref{eq:sizeforward} and \eqref{eq:sizereverse}, the finite-cover
comparison for $M_q$, and \eqref{eq:lorentz-general-decomp} yield
\[
 \begin{split}
 \|M_{q,r}f\|_{\rho,\tau}
 &\le C\bigl(\|M_{2,r}g\|_{\rho,\tau}
       +D_{\rho,\tau}\bigr),\\
 \|M_{2,r}g\|_{\rho,\tau}
 &\le C\bigl(\|M_{q,r}f\|_{\rho,\tau}
       +D_{\rho,\tau}\bigr).
 \end{split}
\]
Together with $T_f=T_g+T_h$ and
\eqref{eq:lorentz-remainder}, these prove the operator and local-size
equivalences. Finally,
$a_rf=a_rg+a_rh,\qquad b_f=b_g+b_h$
give the two scalar equivalences. Since
$G_{q,r}f\le M_{q,r}f,$
the distance term is absorbed on the local-size side of
\eqref{eq:lorentznorm}. This proves all joint estimates.
\end{proof}

\subsection{A link with different Fock exponents}
\begin{lemma}
\label{lem:local-one-descending}
Assume \eqref{eq:weight}, fix $r>0$, let $1\le q<p<\infty$, and put
$s=pq/(p-q)$. If $f\in L^1_{\rm loc}$ and $M_{1,r}f\in L^s$, then
$f\in\Symb$ and the initial Toeplitz operator extends to a compact map
$F^p_\varphi\to F^q_\varphi$. Moreover,
$\norm{T_f}_{p\to q}\le C\norm{M_{1,r}f}_s.$
\end{lemma}

\begin{proof}
Partition $\C^n$ into congruent half-open cubes $Q_j$, with centers
$a_j$ and diameter less than $r/2$. For $u\in F^p_\varphi$, set
\[
 m_j=\int_{Q_j}|f|\,dv,\qquad
 u_j=\sup_{w\in Q_j}|u(w)|e^{-\varphi(w)}.
\]
For $z\in Q_j$, we have $Q_j\subset B(z,r)$, and hence
$m_j\le C M_{1,r}f(z)$. Raising to the power $s$ and integrating over
$Q_j$ gives
$\norm{m}_{\ell^s}\le C\norm{M_{1,r}f}_s.$
The weighted holomorphic submean inequality and bounded overlap of
fixed enlargements of the cubes give
$\norm{(u_j)}_{\ell^p}\le C\norm{u}_{p,\varphi}.$

For $u\in F^p_\varphi$, consider the kernel integral
\[
 \widetilde T_fu(z)
 =\int_{\C^n}f(w)u(w)K(z,w)e^{-2\varphi(w)}\,dv(w).
\]
The kernel estimate gives
\[
 |\widetilde T_fu(z)|e^{-\varphi(z)}
 \le C\sum_j e^{-c|z-a_j|}m_ju_j.
\]
Since $m\in\ell^s$ and $(u_j)\in\ell^p$, both sequences are bounded,
and the exponential lattice sum is uniformly finite. Thus the integral
converges absolutely and locally uniformly in $z$, so
$\widetilde T_fu$ is entire. Discrete Young's inequality for the
exponential synthesis operator, followed by H\"older's inequality and
$1/q=1/p+1/s$, gives
\[
 \norm{\widetilde T_fu}_{q,\varphi}
 \le C\norm{(m_ju_j)}_{\ell^q}
 \le C\norm{m}_{\ell^s}\norm{(u_j)}_{\ell^p}
 \le C\norm{M_{1,r}f}_s\norm{u}_{p,\varphi}.
\]

For each $\zeta$, kernel decay and boundedness of $m$ give
\[
 \int_{\C^n}|f(w)k_\zeta(w)|e^{-\varphi(w)}\,dv(w)<\infty.
\]
Hence $f\in\Symb$. Therefore $\widetilde T_f$ agrees on $\Gamma$
with the initial Toeplitz operator and gives the required bounded
extension.

For a finite set $J$ of cube indices, let
$E_J=\bigcup_{j\in J}Q_j$. The same estimate yields
\[
 \norm{T_f-T_{f\chi_{E_J}}}_{p\to q}
 \le C\norm{(m_j)_{j\notin J}}_{\ell^s}\longrightarrow0
\]
as $J$ increases to the full lattice. It remains to show that each
$T_{f\chi_{E_J}}$ is compact.

Since $E_J$ is bounded,
\[
 T_{f\chi_{E_J}}
 =\int_{E_J} f(w)e^{-2\varphi(w)}
 \bigl[u\mapsto u(w)K(\cdot,w)\bigr]\,dv(w).
\]
The rank-one factor
\[
 w\mapsto\bigl[u\mapsto u(w)K(\cdot,w)\bigr]
\]
is bounded and continuous in operator norm from
$F^p_\varphi$ to $F^q_\varphi$ on the closure of $E_J$.
Indeed, Cauchy's estimate and the local submean bound give
operator-norm continuity of evaluations. Continuity of $K(\cdot,w)$ in
$F^q_\varphi$ follows from the kernel bounds and dominated convergence.
Thus the operator-valued integral is approximable in operator norm by
finite sums of rank-one operators, because
$f\in L^1(E_J)$. Hence $T_{f\chi_{E_J}}$ is compact.

The norm limit of these compact truncations is $T_f$, so $T_f$ is
compact and the stated norm estimate follows.
\end{proof}

\begin{proof}[Proof of Corollary~\ref{cor:unified-descending}]
By Lemma~\ref{lem:mixed-diagonal}, assertion \textup{(i)} implies
$\norm{b_f}_s\le C\norm{T_f^{p,q}}_{p\to q}.$
Applying Theorem~\ref{thm:lorentz} with $(\rho,\tau)=(s,s)$
and local exponent $\sigma$ shows that
\textup{(iii)}--\textup{(vi)} are equivalent and gives
\[
 \norm{M_{\sigma,r}f}_s
 \le C\bigl(\norm{b_f}_s+D\bigr).
\]
Conversely, $M_{1,r}f\le M_{\sigma,r}f$, so
Lemma~\ref{lem:local-one-descending} gives
\textup{(vi)}$\Rightarrow$\textup{(ii)} and
\[
 \norm{T_f^{p,q}}_{p\to q}
 \le C\norm{M_{\sigma,r}f}_s.
\]
Since \textup{(ii)} implies \textup{(i)}, all six assertions are
equivalent.

Moreover, $G_{\sigma,r}f\le M_{\sigma,r}f$, and hence
$D\le\norm{M_{\sigma,r}f}_s$. Combining the preceding estimates with
the joint norm comparisons in
Theorem~\ref{thm:lorentz} proves
\eqref{eq:unifiednorm}. The operators $T_f^{p,q}$ and $T_f^{2,2}$
are extensions of the same initial Toeplitz operator on $\Gamma$,
but act on different spaces.
\end{proof}

\subsection{Counting functions}

For $0<\rho<\infty$ and compact operators between Hilbert spaces, write
\[
 \Sch_{\rho,\infty}^{0}
 =\{A:\lim_{k\to\infty}k^{1/\rho}s_k(A)=0\}.
\]
Equivalently, $A\in\Sch_{\rho,\infty}^{0}$ if
$N(t;A)=o(t^{-\rho})$ as $t\downarrow0$.

\begin{proposition}[Counting functions and the analytic-distance remainder]
\label{prop:counting-transfer}
Under the assumptions of Theorem~\ref{thm:lorentz}, suppose one of
its equivalent conditions holds. Use its decomposition $f=g+h$, and
put
$P_g=M_g^*M_g=T_{|g|^2},\qquad R=H_g-JT_h.$
Then $R\in\Sch_{\rho,\tau}$,
$\|R\|_{\Sch_{\rho,\tau}}\le CD_{\rho,\tau},$
and, for all $\varepsilon,t>0$,
\[
 \begin{split}
 N((1+\varepsilon)t;T_f)
 &\le N(t^2;P_g)+N(\varepsilon t;R),\\
 N((1+\varepsilon)^2t^2;P_g)
 &\le N(t;T_f)+N(\varepsilon t;R).
 \end{split}
\]
Moreover,
\[
 N(\varepsilon t;R)
 \le C(\varepsilon t)^{-\rho}D_{\rho,\tau}^{\rho}.
\]

For the small cubes $Q_j$ of Lemma~\ref{lem:positive-lorentz}, put
\[
 m_j=\int_{Q_j}|g|^2\,dv,\qquad
 D_g(t)=\#\{j:m_j>t^2\}.
\]
There are constants $c,C>0$ such that
\begin{equation}\label{eq:complexcounts}
 \begin{split}
 D_g(t)
 &\le C N(ct;T_f)+C N(ct;R),\\
 N(Ct;T_f)
 &\le C_\eta t^{-2\eta}
    \sum_{m_j>t^2}m_j^\eta+N(t;R),
 \qquad 0<\eta<\min\{\rho/2,1\}.
 \end{split}
\end{equation}

If $0<\tau<\infty$, then, for every fixed $\varepsilon>0$,
\[
 N(\varepsilon t;R)=o(t^{-\rho})
 \qquad(t\downarrow0),
\]
and
\begin{equation}\label{eq:lorentz-remainder-distribution}
 \left(\int_0^\infty
 [tN(t;R)^{1/\rho}]^\tau\,\frac{dt}{t}\right)^{1/\tau}
 \le CD_{\rho,\tau}.
\end{equation}

For an exact leading term, suppose separately that on $\Gamma$
$M_g=JT_f+R,$
where $M_g$ has a compact extension and
$R\in\Sch_{\rho,\infty}^{0}$. Then $T_f$ has a compact extension, and
\begin{equation}\label{eq:exact-counting-transfer}
 N(t^2;M_g^*M_g)\sim Lt^{-\rho},\quad L>0
 \quad\Longrightarrow\quad
 N(t;T_f)\sim Lt^{-\rho}.
\end{equation}
For the decomposition above, it is enough to assume
\begin{equation}\label{eq:mixed-counting-hypotheses}
 M_{q,r}f\in L^{\rho,\infty},\qquad
 \lim_{u\to\infty}u^{1/\rho}(G_{q,r}f)^*(u)=0,
\end{equation}
together with the standing weight and domain assumptions.
The first condition implies $G_{q,r}f\in L^{\rho,\infty}$, since
$G_{q,r}f\le M_{q,r}f$. The second condition follows, in particular,
from
\[
 G_{q,r}f\in L^{\rho,\tau_0}
 \qquad\text{for some }0<\tau_0<\infty.
\]
Neither $M_g$ nor $T_f$ is required to belong to
$\Sch_{\rho,\infty}^{0}$.
\end{proposition}

\begin{proof}
On $\Gamma$,
$M_g=JT_f+R,$
and the identity extends to $F^2_\varphi$. Since
$J:F^2_\varphi\to L^2_\varphi$ is isometric,
\[
 s_j(JT_f)=s_j(T_f),\qquad
 s_j(M_g)^2=s_j(P_g).
\]
Applying
$N(a+b;A+B)\le N(a;A)+N(b;B)$
to $JT_f=M_g-R$ and to $M_g=JT_f+R$ gives the two displayed counting
inequalities. By
$\Sch_{\rho,\tau}\hookrightarrow\Sch_{\rho,\infty},$
\[
 N(\varepsilon t;R)
 \le C(\varepsilon t)^{-\rho}
    \|R\|_{\Sch_{\rho,\tau}}^\rho.
\]
The Lorentz Hankel estimate and
\eqref{eq:lorentz-remainder} therefore give
\[
 \|R\|_{\Sch_{\rho,\tau}}
 \le CD_{\rho,\tau}.
\]

We next prove \eqref{eq:complexcounts}. By
\eqref{eq:positivecountlower},
$D_g(t)\le C N(ct^2;P_g)=C N(c^{1/2}t;M_g).$
Using
$N(2t;M_g)\le N(t;T_f)+N(t;R)$
and changing constants gives the first estimate in
\eqref{eq:complexcounts}.

For the second, let $C_0$ be the threshold constant in
\eqref{eq:positivecountupper}. Then
\[
 N((\sqrt{C_0}+1)t;T_f)
 \le N(\sqrt{C_0}t;M_g)+N(t;R)
 =N(C_0t^2;P_g)+N(t;R).
\]
Applying \eqref{eq:positivecountupper} with threshold $t^2$ yields
the required upper bound.

Formula \eqref{eq:lorentz-remainder-distribution} is the distribution
form of the Lorentz--Schatten quasi-norm. If $\tau<\infty$, then
$k^{1/\rho}s_k(R)\longrightarrow0.$
Equivalently,
$N(t;R)=o(t^{-\rho})\qquad(t\downarrow0),$
and hence the same holds with $t$ replaced by $\varepsilon t$.

For the separate exact-asymptotic assertion,
$M_g-R$ is compact. Its range is contained in the closed subspace
$JF^2_\varphi$, because this is true on the dense domain $\Gamma$.
Thus it defines a compact operator $JT_f$.
Since $R\in\Sch_{\rho,\infty}^{0}$,
\[
 N(\varepsilon t;R)=o(t^{-\rho})
 \qquad(t\downarrow0)
\]
for every fixed $\varepsilon>0$. The preceding counting inequalities
therefore imply
\[
 L(1+\varepsilon)^{-\rho}
 \le \liminf_{t\downarrow0}t^\rho N(t;T_f)
 \le \limsup_{t\downarrow0}t^\rho N(t;T_f)
 \le L(1+\varepsilon)^\rho.
\]
Letting $\varepsilon\downarrow0$ proves
\eqref{eq:exact-counting-transfer}.

It remains to verify the sufficient symbol conditions in
\eqref{eq:mixed-counting-hypotheses}. We first note that if
$v\in L^{a,\infty}$ and
$u^{1/a}v^*(u)\to0$ as $u\to\infty$,
then this property is preserved by pointwise domination, translations,
and finite sums. For the last assertion,
\[
 \left(\sum_{\nu=1}^{\ell}v_\nu\right)^*(u)
 \le\sum_{\nu=1}^{\ell}v_\nu^*(u/\ell).
\]
The finite-cover argument in Lemma~\ref{lem:lorentz-radius} gives,
for every fixed $R>0$, a pointwise bound of $G_{q,R}f$ by a finite
sum of translates of $G_{q,r}f$. Hence the second condition in
\eqref{eq:mixed-counting-hypotheses} holds at every fixed radius.

We also need the corresponding statement for positive Toeplitz
operators. Suppose
$\mu(B(\cdot,s))\in L^{a,\infty}$
and $u^{1/a}(\mu(B(\cdot,s)))^*(u)\to0$ as $u\to\infty$.
Choose the cubes of Lemma~\ref{lem:positive-lorentz} with diameter
less than $s$, and put $x_j=\mu(Q_j)$. For $z\in Q_j$,
$x_j\le\mu(B(z,s)).$
Since the cubes have equal volume,
$k^{1/a}x_k^*\longrightarrow0.$
The positive criterion gives compactness of $T_\mu$. For
$0<\eta<\min\{a,1\}$,
\eqref{eq:positive-rearrangement-upper} gives
\[
 k^{\eta/a}s_k(T_\mu)^\eta
 \le C k^{\eta/a-1}
    \sum_{j\le k}(x_j^*)^\eta.
\]
Given $\epsilon>0$, choose $J$ so that
$x_j^*\le\epsilon j^{-1/a},\qquad j>J.$
Since $\eta<a$,
\[
 k^{\eta/a-1}\sum_{j\le J}(x_j^*)^\eta\longrightarrow0,
\]
while
\[
 k^{\eta/a-1}
 \sum_{J<j\le k}(x_j^*)^\eta
 \le C_{a,\eta}\epsilon^\eta.
\]
Thus
$k^{1/a}s_k(T_\mu)\longrightarrow0,$
that is,
$T_\mu\in\Sch_{a,\infty}^{0}.$

Under \eqref{eq:mixed-counting-hypotheses},
Theorem~\ref{thm:lorentz} with $(\rho,\infty)$ gives compact
extensions of $T_f$ and $M_g$. Use the decomposition $f=g+h$ at a
fixed sufficiently small auxiliary radius $t>0$.
The pointwise estimates
\eqref{eq:q-decomp-error} and \eqref{eq:q-to-two}, together with the
preceding radius comparison, show that
\[
 M_{q,t/2}h
 \quad\text{and}\quad
 G_{2,t}g
\]
satisfy the same rearrangement limit as $G_{q,r}f$.
Lemma~\ref{lem:lorentz-radius}
also gives their boundedness, so the original domain argument applies:
$g,h\in\Symb,\qquad gk_z\in L^2_\varphi.$

The fixed-ball masses of $|h|\,dv$ are bounded by
$CM_{q,t/2}h$. Applying the positive-operator statement with
$a=\rho$ and using the factorization from
\eqref{eq:lorentz-remainder} gives
\[
 s_{2k-1}(T_h)\le s_k(T_{|h|}),\qquad
 T_h\in\Sch_{\rho,\infty}^{0}.
\]

Apply the decomposition in Lemma~\ref{lem:lorentz-hankel} to $g$.
The associated measures
\[
 d\mu_1=|\bar\partial g_1|^2\,dv,\qquad
 d\mu_2=|g_2|^2\,dv
\]
have fixed-ball masses bounded by
$C(G_{2,t}g)^2$
after decreasing the auxiliary radius if necessary. Hence their
local masses satisfy the rearrangement limit above with $a=\rho/2$.
The preceding
positive-operator argument, with $a=\rho/2$, and the form estimates
from Lemma~\ref{lem:lorentz-hankel} give
\[
 s_k(H_{g_j})^2\le C s_k(T_{\mu_j}),\qquad
 H_{g_j}\in\Sch_{\rho,\infty}^{0},
 \quad j=1,2.
\]
Since
$s_{2k-1}(A+B)\le s_k(A)+s_k(B),$
the class $\Sch_{\rho,\infty}^{0}$ is stable under finite sums.
Therefore
$R=H_g-JT_h\in\Sch_{\rho,\infty}^{0}.$

Finally, if $v\in L^{\rho,\tau_0}$ with $0<\tau_0<\infty$, then
monotonicity of $v^*$ gives
\[
 [u^{1/\rho}v^*(u)]^{\tau_0}
 \le C\int_{u/2}^{u}
 [s^{1/\rho}v^*(s)]^{\tau_0}\,\frac{ds}{s}
 \longrightarrow0.
\]
This proves the stated finite-secondary-index sufficient condition.
\end{proof}

\subsection{Regular decay envelopes}

The next result treats decay rates more general than powers. Related positive local-mass and one-sided Hankel estimates in the
radial setting appear in \cite{HWZ}.
We prove the estimates for fixed Euclidean balls. We then combine
them with the local-size estimates for complex Toeplitz symbols.
Each criterion retains the analytic-distance term.

Indeed, in the classical one-dimensional Fock space, let
$0\ne\chi\in C_c^\infty(\C)$ be radial and nonnegative. For every
fixed $q\ge1$ and $r>0$, both $M_{q,r}\chi$ and $G_{q,r}\chi$ have
rearrangements that vanish for sufficiently large arguments, whereas
$T_\chi$ has the strictly positive monomial eigenvalues
\[
 \lambda_k=\frac1{k!}\int_0^\infty
 \chi(\sqrt t)t^ke^{-t}\,dt,\qquad k\ge0.
\]
Thus $T_\chi$ has infinite rank. Hence no two-sided estimate of the form
$s_k(T_f)\asymp (M_{q,r}f)^*(ck)$
can hold in general, even after adding
$(G_{q,r}f)^*(ck)$ to the right-hand side.

Let $\Psi:[1,\infty)\to(0,\infty)$ be nondecreasing and unbounded.
Assume that $\Psi$ has finite upper type: for some $\gamma>0$ and
$C_\Psi\ge1$,
\begin{equation}\label{eq:upper-type}
 \Psi(\lambda t)\le C_\Psi\lambda^\gamma\Psi(t),
 \qquad t\ge1,\quad\lambda\ge1.
\end{equation}
After replacing $\Psi$ by a comparable function, we may also assume
that $\Psi(t)/t^\gamma$ is nonincreasing. Indeed,
\[
 \Psi_0(t)
 =t^\gamma\inf_{1\le s\le t}\frac{\Psi(s)}{s^\gamma}
\]
is nondecreasing,
$C_\Psi^{-1}\Psi(t)\le\Psi_0(t)\le\Psi(t),$
and $\Psi_0(t)/t^\gamma$ is nonincreasing. This is the
monotone-quotient form used in \cite{FWZ,HWZ}.

For a measurable function $v$ and a compact operator $A$, define
\[
 \|v\|_{\mathcal M_\Psi}
 =\sup_{t>0}\Psi(1+t)v^*(t),
 \qquad
 \|A\|_{\mathcal S_\Psi}
 =\sup_{k\ge1}\Psi(k)s_k(A).
\]
The shift $1+t$ controls both boundedness and decay at infinity. For
example, when $\Psi(t)=t^{1/\rho}$,
$\|\cdot\|_{\mathcal M_\Psi}$ is equivalent to the joint
$L^\infty$ and weak-$L^\rho$ gauge.

We use the following consequence of \eqref{eq:upper-type}. Choose
$0<\theta<\min\{1,(2\gamma)^{-1}\}$
and put $\Psi_\theta=\Psi^\theta$. Then
\begin{equation}\label{eq:upper-type-hardy}
 \frac1t\int_0^t\frac{ds}{\Psi_\theta(1+s)}
 \le \frac{C}{\Psi_\theta(1+t)},\qquad t>0.
\end{equation}
Indeed, for $0<s<t$, \eqref{eq:upper-type} gives
\[
 \frac1{\Psi_\theta(1+s)}
 \le C\left(\frac{1+t}{1+s}\right)^{\gamma\theta}
    \frac1{\Psi_\theta(1+t)},
\]
and \eqref{eq:upper-type-hardy} follows from $\gamma\theta<1$.
The smaller choice of $\theta$ also allows the same argument for $\Psi^2$.

Let $U$ be a positive linear operator bounded on both $L^1$ and
$L^\infty$. Calder\'on's rearrangement estimate
\cite[Chapter~2]{BennettSharpley}, together with
\eqref{eq:upper-type-hardy}, gives
\begin{equation}\label{eq:envelope-averaging}
 \|Uu\|_{\mathcal M_{\Psi_\theta}}
 \le C_U\|u\|_{\mathcal M_{\Psi_\theta}}.
\end{equation}
For sequences, use the gauge $\sup_{k\ge1}\Psi(k)^\theta x_k^*$.
The same estimate holds for positive operators on $\ell^1$ and
$\ell^\infty$. Applying
\eqref{eq:envelope-averaging} to $|v|^\theta$ yields
\begin{equation}\label{eq:nonlinear-envelope-averaging}
 \left\|(U|v|^\theta)^{1/\theta}\right\|_{\mathcal M_\Psi}
 \le C_U\|v\|_{\mathcal M_\Psi}.
\end{equation}

We shall refer to \eqref{eq:nonlinear-envelope-averaging} as the
upper-type averaging lemma. It applies to the fixed-ball averages used
below. Exponentially weighted sums and finite sums of translates are
handled after taking the $\theta$th power. The same argument, combined
with the finite-cover and decomposition estimates from
Lemma~\ref{lem:lorentz-radius}, gives fixed-radius equivalence for
$M_{q,r}$ and $G_{q,r}$ on $\mathcal M_\Psi$.

\begin{theorem}[Regular singular-value decay]
\label{thm:regular-decay}
Assume \eqref{eq:weight}. Let $\Psi$ be as above and satisfy
\eqref{eq:upper-type}. Let $1\le q<\infty$, $r>0$, and
$f\in\Symb\cap L^q_{\rm loc}$. Set $\|T_f\|_{\mathcal S_\Psi}=+\infty$
if the initial operator has no compact extension to $F^2_\varphi$.
Then, with function gauges allowed to take the value $+\infty$,
\begin{equation}\label{eq:regular-decay-equivalence}
 \begin{split}
 \|T_f\|_{\mathcal S_\Psi}+\|G_{q,r}f\|_{\mathcal M_\Psi}
 &\asymp\|M_{q,r}f\|_{\mathcal M_\Psi}\\
 &\asymp\|a_rf\|_{\mathcal M_\Psi}
    +\|G_{q,r}f\|_{\mathcal M_\Psi}\\
 &\asymp\|b_f\|_{\mathcal M_\Psi}
    +\|G_{q,r}f\|_{\mathcal M_\Psi}.
 \end{split}
\end{equation}
The constants depend on the geometric parameters, $q,r$, and the
upper-type data of $\Psi$, but not on $f$.
\end{theorem}

Equivalently, $(M_{q,r}f)^*(t)=O(\Psi(1+t)^{-1})$ if and only if
\[
 s_k(T_f)=O(\Psi(k)^{-1})
 \quad\text{and}\quad
 (G_{q,r}f)^*(t)=O(\Psi(1+t)^{-1}),
\]
where the function bounds are uniform for $t>0$ and the sequence
bound is uniform for $k\ge1$. The local-size condition may be replaced
by the corresponding joint condition involving $a_rf$ or $b_f$. The choice
$\Psi(t)=t^{1/\rho}$ gives the weak Schatten endpoint. The theorem
also applies to
$\Psi(t)=[\log(e+t)]^\beta\quad(\beta>0)$
and to
\[
 \Psi(t)=t^\alpha[\log(e+t)]^\beta,
 \qquad \alpha>0,\quad\beta\in\mathbb R,
\]
after a monotone modification on a bounded interval. The case
$\alpha=0$ is included when $\beta>0$.

\begin{proof}
We first prove the positive counterpart. Let $\mu$ be a positive
locally finite measure, let $Q_j$ be the cubes from
Lemma~\ref{lem:positive-lorentz}, and put $x_j=\mu(Q_j)$. Then
\begin{equation}\label{eq:positive-envelope}
 \|T_\mu\|_{\mathcal S_\Psi}
 \asymp\sup_{k\ge1}\Psi(k)x_k^*
 \asymp\|\mu(B(\cdot,r))\|_{\mathcal M_\Psi}
 \asymp\|\widetilde\mu\|_{\mathcal M_\Psi}.
\end{equation}
Indeed, \eqref{eq:positivecountlower} gives
$x_j^*\le C s_{\lceil j/L\rceil}(T_\mu),\qquad j\ge1.$
Since $j\le L\lceil j/L\rceil$, \eqref{eq:upper-type} yields the
corresponding lower estimate in \eqref{eq:positive-envelope}.

For the reverse estimate, assume $0<W=\sup_k\Psi(k)x_k^*<\infty$
and normalize $W=1$; the case $W=0$ is immediate.
Then $x_k^*\le\Psi(k)^{-1}\to0$. The positive Carleson criterion
gives boundedness, and truncating to finitely many cubes gives
compactness: the norm of the omitted part is bounded by its largest
cube mass. Choose
$0<\eta<\min\{1,\gamma^{-1}\}.$
The threshold decomposition used in
Lemma~\ref{lem:positive-lorentz} gives
\[
 s_k(T_\mu)\le C
 \left(\frac1k\sum_{j\le k}(x_j^*)^\eta\right)^{1/\eta},
\]
The high-mass part uses only finitely many cubes, so its strong
$\mathcal S_\eta$ estimate requires no Lorentz hypothesis. If
$x_k^*=0$, the whole measure has fewer than $k$ positive cube masses,
and the same strong estimate gives the displayed bound. For $j\le k$,
\eqref{eq:upper-type} gives
\[
 \Psi(j)^{-\eta}
 \le C(k/j)^{\gamma\eta}\Psi(k)^{-\eta}.
\]
Since $\gamma\eta<1$,
\[
 \sum_{j\le k}\Psi(j)^{-\eta}
 \le Ck\Psi(k)^{-\eta},
\]
and hence
\[
 \sup_k\Psi(k)s_k(T_\mu)
 \le C\sup_k\Psi(k)x_k^*.
\]
Finite coverings give the local-mass equivalence in
\eqref{eq:positive-envelope}. Moreover, for $z\in Q_j$,
\[
 cx_j\le\widetilde\mu(z)
 \le C\sum_i e^{-c|z_i-z_j|}x_i.
\]
After taking the $\theta$th power, the sequence form of
\eqref{eq:envelope-averaging} controls the exponential convolution.
This proves \eqref{eq:positive-envelope}.

Suppose first that
$\|M_{q,r}f\|_{\mathcal M_\Psi}<\infty.$
For $d\mu=|f|\,dv$,
$\mu(B(z,r))\le C M_{q,r}f(z).$
Hence \eqref{eq:positive-envelope} gives
$T_\mu\in\mathcal S_\Psi$. Let
$A:F^2_\varphi\to L^2_\varphi$ be multiplication by $|f|^{1/2}$.
Then $A^*A=T_\mu$. With
$\operatorname{sgn}f=f/|f|$ on $\{f\ne0\}$ and zero elsewhere,
$T_f=A^*M_{\operatorname{sgn}f}A$
on $\Gamma$. Thus
\[
 s_{2k-1}(T_f)
 \le s_k(A)^2=s_k(T_\mu).
\]
Since $\Psi(2k)\le C\Psi(k)$,
\[
 \|T_f\|_{\mathcal S_\Psi}
 \le C\|M_{q,r}f\|_{\mathcal M_\Psi}.
\]
Also $G_{q,r}f\le M_{q,r}f$, and therefore
\[
\|T_f\|_{\mathcal S_\Psi}
 +\|G_{q,r}f\|_{\mathcal M_\Psi}
 \le C\|M_{q,r}f\|_{\mathcal M_\Psi}.
\]

For the reverse estimates, assume
$D_\Psi=\|G_{q,r}f\|_{\mathcal M_\Psi}<\infty$.
We first construct the decomposition without assuming $T_f$ bounded.
Use Lemma~\ref{lem:q-decomposition} to write $f=g+h$. Fixed-radius equivalence for
the envelope gauge gives, at a sufficiently small fixed radius $r_0$,
\[
 G_{2,r_0}g+M_{q,r_0/2}h\le C\Delta,
 \qquad
 \Delta=G_{q,16r_0}f,
\]
with
$\|\Delta\|_{\mathcal M_\Psi}\le CD_\Psi.$
Finiteness of the envelope gauge also gives the boundedness required
in the domain argument; hence $g,h\in\Symb$ and
$gk_z\in L^2_\varphi$.

Apply the argument of Lemma~\ref{lem:lorentz-hankel} to $g$.
The two positive measures arising there have fixed-ball masses bounded
by $C\Delta^2$. Applying \eqref{eq:positive-envelope} to $\Psi^2$, which has upper
type $2\gamma$, and using the form inequalities gives
\[
 H_g\in\mathcal S_\Psi,\qquad
 \|H_g\|_{\mathcal S_\Psi}
 \le C\|\Delta\|_{\mathcal M_\Psi}.
\]
Likewise, the fixed-ball masses of $|h|\,dv$ are controlled by
$CM_{q,r_0/2}h$. The positive factorization gives
\[
 T_h\in\mathcal S_\Psi,\qquad
 \|T_h\|_{\mathcal S_\Psi}
 \le C\|\Delta\|_{\mathcal M_\Psi}.
\]
Indeed,
\[
 s_{2k-1}(T_h)\le s_k(T_{|h|}),
 \qquad \Psi(2k)\le C\Psi(k).
\]

If also $\|T_f\|_{\mathcal S_\Psi}<\infty$, the identity
$M_g=JT_f+(H_g-JT_h)$ on $\Gamma$ gives a compact extension of
$M_g$. The closure argument after \eqref{eq:column} identifies it
with multiplication. Ky Fan's inequality and the doubling consequence of
\eqref{eq:upper-type} give
\[
 \|M_g\|_{\mathcal S_\Psi}
 \le C\bigl(\|T_f\|_{\mathcal S_\Psi}+D_\Psi\bigr).
\]
Now
$M_g^*M_g=T_{|g|^2}.$
The local mass of $|g|^2dv$ is $|B(0,r)|(M_{2,r}g)^2$.
Apply \eqref{eq:positive-envelope} to $\Psi^2$ and use
$(v^2)^*=(v^*)^2$ for $v\ge0$, together with
$s_k(M_g^*M_g)=s_k(M_g)^2$. This gives
\begin{equation}\label{eq:envelope-g-local}
 \|M_{2,r}g\|_{\mathcal M_\Psi}
 \le C\bigl(\|T_f\|_{\mathcal S_\Psi}+D_\Psi\bigr).
\end{equation}

The pointwise estimates
\eqref{eq:sizeforward}--\eqref{eq:sizereverse}, together with the
upper-type averaging lemma and fixed-radius comparison, yield
\[
 \begin{split}
 \|M_{q,r}f\|_{\mathcal M_\Psi}
 &\le C\bigl(\|M_{2,r}g\|_{\mathcal M_\Psi}+D_\Psi\bigr),\\
 \|M_{2,r}g\|_{\mathcal M_\Psi}
 &\le C\bigl(\|M_{q,r}f\|_{\mathcal M_\Psi}+D_\Psi\bigr).
 \end{split}
\]
Combining the first estimate with
\eqref{eq:envelope-g-local} proves the reverse operator/local-size
bound in \eqref{eq:regular-decay-equivalence}.

It remains to compare the scalar tests. The direct estimates
$|a_rf|\le M_{q,r}f$
and
\[
 |b_f(z)|\le C\sum_{k\in\mathbb Z^{2n}}
 e^{-c|k|}M_{q,r}f(z+\delta k)
\]
together with the upper-type averaging lemma give
\[
 \|a_rf\|_{\mathcal M_\Psi}
 +\|b_f\|_{\mathcal M_\Psi}
 \le C\|M_{q,r}f\|_{\mathcal M_\Psi}.
\]
Conversely, apply \eqref{eq:meanrecovery} with exponent $\theta$.
Then \eqref{eq:nonlinear-envelope-averaging} and fixed-radius
equivalence give
\[
 \|M_{q,r}f\|_{\mathcal M_\Psi}
 \le C\bigl(\|a_rf\|_{\mathcal M_\Psi}
       +\|G_{q,r}f\|_{\mathcal M_\Psi}\bigr).
\]

For the Berezin converse, assume only $D_\Psi<\infty$ and
$\|b_f\|_{\mathcal M_\Psi}<\infty$. Use the same distance-based
decomposition, and let $d\mu_h=|h|\,dv$. Its fixed-ball masses are
bounded by $C\Delta$, so \eqref{eq:positive-envelope} and
$|b_h|\le\widetilde\mu_h$ give
\[
 \|b_h\|_{\mathcal M_\Psi}
 \le C\|\Delta\|_{\mathcal M_\Psi}.
\]
The two measures arising from the Hankel estimate for $g$ have
fixed-ball masses bounded by $C\Delta^2$. Applying
\eqref{eq:positive-envelope} to $\Psi^2$ and testing the corresponding
form inequalities on $k_z$ gives
\[
 \|E_g\|_{\mathcal M_\Psi}
 \le C\|\Delta\|_{\mathcal M_\Psi},
 \qquad
 E_g(z)=\|H_gk_z\|_{2,\varphi}.
\]
Since
$b_g=b_f-b_h,$
the local-$L^2$ size estimate used before
\eqref{eq:lorentz-local-recovery}, together with the upper-type
averaging lemma, yields
\[
 \|M_{2,r}g\|_{\mathcal M_\Psi}
 \le C\bigl(\|b_f\|_{\mathcal M_\Psi}+D_\Psi\bigr).
\]
Finally, \eqref{eq:sizeforward} and fixed-radius comparison give
\[
 \|M_{q,r}f\|_{\mathcal M_\Psi}
 \le C\bigl(\|b_f\|_{\mathcal M_\Psi}+D_\Psi\bigr).
\]
This proves \eqref{eq:regular-decay-equivalence}.
\end{proof}

\section{Weyl laws for quadratic weights and a sharp limitation}
\label{sec:weyl}
The counting comparison applies to complex symbols with a nonradial
principal part. In the standard model, we compute the leading constant
by the anti-Wick/Weyl correspondence. Unitary maps
then give the result for every strictly convex quadratic weight.

\begin{theorem}[The standard quadratic model]
\label{thm:classical-weyl}
Let $\varphi(z)=|z|^2/2$ and let $m>0$. Identify $\C^n$ with
$\R^{2n}$ and write $\langle w\rangle=(1+|w|^2)^{1/2}$ for Euclidean
vectors $w$. Suppose that
$f$ is a complex-valued symbol of order $-m$ with an asymptotically
homogeneous principal part in the following sense: for some
$\delta>0$, some
$f_{-m}\in C^\infty(S^{2n-1})$, and a cut-off
$\chi\in C^\infty(\R^{2n})$ which vanishes near the origin and equals
one for $|z|$ large,
\[
 f(z)=\chi(z)|z|^{-m}f_{-m}(z/|z|)+r(z),
 \qquad
 |\partial^\alpha r(z)|\le C_\alpha
    \langle z\rangle^{-m-\delta-|\alpha|}
\]
for every real multi-index $\alpha$. Assume $f_{-m}\not\equiv0$ and set
\[
 \rho=\frac{2n}{m},\qquad
 C_f=\frac1{2n\pi^n}\int_{S^{2n-1}}
       |f_{-m}(\omega)|^\rho\,dS(\omega).
\]
Here $dS$ is standard Euclidean surface measure.
Then the Toeplitz operator $T_f$ on $F^2_\varphi$ satisfies
\begin{equation}\label{eq:classical-weyl}
 N(t;T_f)\sim C_ft^{-\rho}\quad(t\downarrow0),
 \qquad
 s_j(T_f)\sim C_f^{1/\rho}j^{-1/\rho}\quad(j\to\infty).
\end{equation}
No ellipticity is assumed: $f_{-m}$ may vanish on open subsets, or on a
set of positive surface measure; only $f_{-m}\not\equiv0$ is required.
\end{theorem}

\begin{proof}
For every fixed $r_0>0$, the symbol estimates and the constant local
holomorphic approximant $f(z)$ give
\[
 M_{2,r_0}f(z)=O(\langle z\rangle^{-m}),
 \qquad
 G_{2,r_0}f(z)=O(\langle z\rangle^{-m-1}).
\]
Thus $M_{2,r_0}f\in L^{\rho,\infty}$ and
$G_{2,r_0}f\in L^\rho$, because $(m+1)\rho>2n$. In particular,
$H_f\in\Sch_\rho$ by Lemma~\ref{lem:lorentz-hankel} with
$\tau=\rho$.

We calculate the positive comparison operator
$P_f=M_f^*M_f=T_{|f|^2}$. Replace $dv$ by $\pi^{-n}dv$
simultaneously in the ambient and Fock inner products. The projection,
$P_f$, and its singular values are unchanged. With this normalization,
let $\mathcal U:L^2(\R^n)\to F^2_\varphi$ be the Bargmann transform
\[
 (\mathcal Uu)(z)=\pi^{-n/4}\int_{\R^n}
 e^{-\frac12(|x|^2+z^2)+\sqrt2\,z\cdot x}u(x)\,dx.
\]
Here $z^2=\sum_{j=1}^n z_j^2$ is the holomorphic quadratic form.
For $X=(x,\xi)\in\R^{2n}$, set
\[
 b(X)=\left|f\left(\frac{x-i\xi}{\sqrt2}\right)\right|^2,
 \qquad
 a(X)=\pi^{-n}\int_{\R^{2n}}e^{-|X-Y|^2}b(Y)\,dY.
\]
The anti-Wick/Weyl correspondence \cite[Chapters 1--2]{FollandPhaseSpace}
gives
\[
 \mathcal U^{-1}P_f\mathcal U=\operatorname{Op}^{\rm w}(a),
 \qquad a=e^{\Delta/4}b,
\]
where, for a Schwartz function $u$, the Weyl integral is interpreted
in the oscillatory sense:
\[
 (\operatorname{Op}^{\rm w}(a)u)(x)
 =(2\pi)^{-n}\iint_{\R^n\times\R^n} e^{i(x-y)\cdot\xi}
 a\left(\frac{x+y}{2},\xi\right)u(y)\,dy\,d\xi.
\]
Choose a smooth cut-off $\widetilde\chi$ which vanishes near the origin
and equals one near infinity, and put
\[
 b_{-2m}(X)=\widetilde\chi(X)2^m|X|^{-2m}
 \left|f_{-m}\left(\frac{x-i\xi}{|X|}\right)\right|^2.
\]
The symbol assumptions give a remainder of order $-2m-\delta$ in
$b-b_{-2m}$, while the heat expansion gives a remainder of order
$-2m-2$ in $a-b$. Consequently, with $\eta=\min\{\delta,2\}$, for every
real multi-index $\alpha$,
\begin{equation}\label{eq:weyl-principal}
 |\partial^\alpha(a-b_{-2m})(X)|
 \le C_\alpha\langle X\rangle^{-2m-\eta-|\alpha|}.
\end{equation}
In particular,
\[
 a(r\theta)=2^m r^{-2m}
 |f_{-m}(\theta_x-i\theta_\xi)|^2+O(r^{-2m-\eta}),
\]
and \eqref{eq:weyl-principal} gives the corresponding estimates after
arbitrary derivatives.

The operator $P_f$ is nonnegative and selfadjoint, while
\eqref{eq:weyl-principal} places its real Weyl symbol in the required
negative-order symbol class with the displayed homogeneous principal
part. Apply the negative-order Weyl formula of Dauge and Robert
\cite[Section~2.A, third case, and Remark~1.6(2)]{DaugeRobertWeyl}.
Their argument also treats lower-order perturbations. No ellipticity of
the homogeneous principal symbol is needed. With
the Weyl normalization displayed above, \eqref{eq:weyl-principal} gives
\begin{align*}
 N(\lambda;P_f)
 &\sim \frac{(2\pi)^{-n}}{2n}\lambda^{-n/m}
 \int_{S^{2n-1}}
 \left(2^m|f_{-m}(\theta_x-i\theta_\xi)|^2\right)^{n/m}
 \,dS(\theta)\\
 &=\frac1{2n\pi^n}\lambda^{-n/m}
 \int_{S^{2n-1}}|f_{-m}(\omega)|^{2n/m}\,dS(\omega)
 =C_f\lambda^{-n/m}.
\end{align*}
Here $(\theta_x,\theta_\xi)\mapsto\theta_x-i\theta_\xi$ is an
orthogonal identification of $S^{2n-1}$ with the unit sphere in
$\C^n$. Taking $\lambda=t^2$ gives
$N(t^2;P_f)\sim C_ft^{-\rho}$. Since
$P_f=M_f^*M_f$ is compact, so is $M_f$. Also, $H_f\in\Sch_\rho$ gives
$N(\varepsilon t;H_f)=o(t^{-\rho})$ for every fixed
$\varepsilon>0$. Apply the separate transfer assertion
\eqref{eq:exact-counting-transfer} with $g=f$, $h=0$, and $R=H_f$.
It gives the first asymptotic in \eqref{eq:classical-weyl}.
Since $C_f>0$, monotone inversion gives the singular-value formula.
\end{proof}

Every strictly convex real quadratic weight on $\C^n$ can be written,
up to an additive real constant, as a positive Hermitian quadratic form
plus the real part of a holomorphic polynomial of degree at most two.
Thus the following theorem covers the full strictly convex quadratic
class.

\begin{theorem}[Quadratic weighted Weyl law]
\label{thm:quadratic-weyl}
Let $Q=Q^*>0$ be an $n\times n$ Hermitian matrix, let $c\in\R$, and
let $\psi$ be a holomorphic polynomial of degree at most two. Suppose
that the real quadratic function
$\varphi_Q(z)=\tfrac12\langle Qz,z\rangle+\operatorname{Re}\psi(z)+c$
is strictly convex. Let $f$ satisfy the symbol assumptions of
Theorem~\ref{thm:classical-weyl}, with order $-m$, principal part
$f_{-m}$, and $\rho=2n/m$. Put
\begin{equation}\label{eq:quadratic-weyl-constant}
 C_{f,Q}=\frac{\det_{\C}Q}{2n\pi^n}
 \int_{S^{2n-1}}|f_{-m}(\omega)|^\rho\,dS(\omega).
\end{equation}
Then the Toeplitz operator on $F^2_{\varphi_Q}$ satisfies
\[
N(t;T_f^{\varphi_Q})\sim C_{f,Q}t^{-\rho},
 \qquad
 s_j(T_f^{\varphi_Q})\sim C_{f,Q}^{1/\rho}j^{-1/\rho}.
\]
The first limit is as $t\downarrow0$ and the second as $j\to\infty$.
\end{theorem}

\begin{proof}
First remove the pluriharmonic term. If $F_Q^2$ denotes the space
with weight $\langle Qz,z\rangle/2$, then
\[
 V:F^2_{\varphi_Q}\longrightarrow F_Q^2,
 \qquad (Vg)(z)=e^{-\psi(z)-c}g(z),
\]
is unitary, as is its ambient $L^2$ extension. It commutes with
multiplication by $f$, and therefore
$VT_f^{\varphi_Q}V^{-1}=T_f^Q$.

Set $w=Q^{1/2}z$ and define
\[
 W:F_Q^2\longrightarrow F^2_{|w|^2/2},
 \qquad
 (Wg)(w)=(\det_{\C}Q)^{-1/2}g(Q^{-1/2}w).
\]
Since $dv(w)=(\det_{\C}Q)\,dv(z)$, this map and its ambient extension are
unitary. Hence
\[
 WT_f^QW^{-1}=T_{\widetilde f},
 \qquad \widetilde f(w)=f(Q^{-1/2}w).
\]
The transformed symbol has principal part
\[
 \widetilde f_{-m}(\theta)
 =|Q^{-1/2}\theta|^{-m}
 f_{-m}\left(\frac{Q^{-1/2}\theta}
 {|Q^{-1/2}\theta|}\right).
\]
Apply Theorem~\ref{thm:classical-weyl} to $\widetilde f$. For every
invertible real linear map $L$ on $\R^{2n}$ and every integrable $h$ on
the sphere,
\[
 \int_{S^{2n-1}}|L\theta|^{-2n}
 h\left(\frac{L\theta}{|L\theta|}\right)dS(\theta)
 =|\det_{\R}L|^{-1}\int_{S^{2n-1}}h(\omega)\,dS(\omega).
\]
Use $L=Q^{-1/2}$ and $m\rho=2n$. Since
$|\det_{\R}Q^{-1/2}|^{-1}=\det_{\C}Q$, the standard coefficient becomes
exactly \eqref{eq:quadratic-weyl-constant}. Unitary equivalence
preserves singular values and counting functions.
\end{proof}

The Hessian bounds in \eqref{eq:weight} are sufficient for all local
and ideal estimates above, but they do not determine a global
phase-space density. The following example shows that this distinction
is sharp, even for a positive radial symbol in one complex dimension.

\begin{proposition}[Uniform Hessian bounds do not force a Weyl constant]
\label{prop:no-general-weyl}
For every $\beta>0$ there exist a smooth radial weight $\varphi$ on
$\C$ satisfying
$\tfrac34 I_2\le\operatorname{Hess}_{\R}\varphi\le\tfrac94 I_2$
and the positive radial symbol
$f(z)=(1+|z|^2)^{-\beta/2}$ such that the sequence
$j^{\beta/2}s_j(T_f)$ has two different positive subsequential limits.
Consequently there is no constant $C$ for which
$N(t;T_f)\sim Ct^{-2/\beta}$.
\end{proposition}

\begin{proof}
Choose $\gamma\in C^\infty(\R;[1,2])$, constant near $-\infty$, with
$|\gamma'|\le1/4$. Fix $\delta>0$ and choose
$\sigma_j\to\infty$ so that $\gamma\equiv a_j$ on
$[\sigma_j-\delta,\sigma_j+\delta]$, where
$a_j\in\{1,2\}$ alternates. Put
\[
 \Phi'(r)=r\gamma(\log r),\qquad \Phi(0)=0,
 \qquad \varphi(z)=\Phi(|z|).
\]
Taking $\gamma=1$ near $-\infty$ makes $\varphi$ smooth at the origin.
The tangential and radial eigenvalues of its real Hessian are
\[
 \frac{\Phi'(r)}r=\gamma(\log r),
 \qquad
 \Phi''(r)=\gamma(\log r)+\gamma'(\log r),
\]
which proves the stated Hessian bounds.

The normalized monomials form an orthonormal basis. Indeed, angular
integration of the Taylor series gives the norm as the sum of its
nonnegative monomial terms. The Taylor polynomials therefore converge
in the Fock norm. If $e_k$ is the normalized monomial of
degree $k$, then $T_fe_k=\lambda_ke_k$, where
\begin{equation}\label{eq:variable-weight-eigenvalue}
 \lambda_k=
\frac{\displaystyle\int_0^\infty f(r)r^{2k+1}e^{-2\Phi(r)}\,dr}
 {\displaystyle\int_0^\infty r^{2k+1}e^{-2\Phi(r)}\,dr}.
\end{equation}
Since $f(z)\to0$, the positive Toeplitz compactness criterion shows
that $T_f$ is compact. The numbers $\lambda_k$ decrease. Indeed, with expectation taken
against the denominator-normalized measure in \eqref{eq:variable-weight-eigenvalue},
\[
 \lambda_{k+1}=\frac{\mathbb E_k(fr^2)}{\mathbb E_k(r^2)}
 \le\mathbb E_k f=\lambda_k,
\]
because $f$ decreases and $r^2$ increases. Thus, in decreasing order,
$s_{k+1}(T_f)=\lambda_k$.

Let $R_j=e^{\sigma_j}$. Choose integers $k_j$ with
$k_j+\tfrac12=a_jR_j^2+O(1)$. The logarithm of the density in the
denominator of \eqref{eq:variable-weight-eigenvalue} is
$h_k(r)=(2k+1)\log r-2\Phi(r).$
For every $k$, one has
\[
 h_k''(r)=-\frac{2k+1}{r^2}
 -2\bigl(\gamma(\log r)+\gamma'(\log r)\bigr)\le-\tfrac32.
\]
Thus $h_k$ has a unique maximizer. If $\widehat r_j$ denotes the
maximizer for $k=k_j$, then
$h_{k_j}'(R_j)=O(R_j^{-1})$. The mean-value theorem and the preceding
curvature bound give
$|\widehat r_j-R_j|=O(R_j^{-1})$. Hence $\widehat r_j$ lies in the
$a_j$-plateau for large $j$, and only now the critical-point equation
gives
$a_j\widehat r_j^2=k_j+\tfrac12.$

On $|r-\widehat r_j|\le1$ one also has
$h_{k_j}''(r)\ge-C$, whereas globally
\[
 h_{k_j}(r)\le h_{k_j}(\widehat r_j)
       -\tfrac34(r-\widehat r_j)^2.
\]
Consequently the denominator in
\eqref{eq:variable-weight-eigenvalue} is at least
$ce^{h_{k_j}(\widehat r_j)}$, and its normalized mass outside
$|r-\widehat r_j|\le R$ is at most $Ce^{-cR^2}$. Take
$R=\widehat r_j^{1/2}$. This interval lies inside the plateau for
large $j$, and the logarithmic derivative of $f$ gives
\[
 \sup_{|r-\widehat r_j|\le R}
 \left|\frac{f(r)}{f(\widehat r_j)}-1\right|\longrightarrow0.
\]
It follows that
\[
 \lambda_{k_j}=f(\widehat r_j)(1+o(1))
 =\widehat r_j^{-\beta}(1+o(1)).
\]
The Gaussian tail contributes
$O(\widehat r_j^\beta e^{-c\widehat r_j})=o(1)$ after division by
$f(\widehat r_j)$. Hence
\[
 (k_j+1)^{\beta/2}s_{k_j+1}(T_f)
 =a_j^{\beta/2}+o(1).
\]
The alternating plateaux give the two limits $1$ and $2^{\beta/2}$.
A counting asymptotic with a fixed coefficient would, by monotone
inversion, force a single limit for $j^{\beta/2}s_j(T_f)$, a
contradiction.
\end{proof}

\section{Examples and limitations}\label{sec:examples}
We now give several examples that clarify the role of the hypotheses
in the preceding results. They also show that analytic distance alone
does not control Toeplitz boundedness or compactness, and indicate the
limits of the criteria proved above.
\begin{example}[A zero BDA seminorm does not imply boundedness]
\label{ex:analytic-part}
For the classical weight $\varphi(z)=|z|^2/2$, let $f(z)=z_1$.
Then $G_{q,r}f=0$ for every $q>0$. However,
$T_fk_z(w)=w_1k_z(w)$. Evaluation at $w=z$ gives
\[
 \norm{T_fk_z}_{q,\varphi}\ge c|z_1|,
 \qquad \sup_z\norm{k_z}_{p,\varphi}<\infty.
\]
Thus $T_f:F^p_\varphi\to F^q_\varphi$ is unbounded for every finite
$p,q\ge1$. This shows why analytic distance alone cannot characterize Toeplitz
boundedness.
\end{example}

\begin{example}[Why the local exponent cannot be replaced by two]
\label{ex:local-q-below-two}
Let $d=2n$ and consider the classical weight $\varphi(z)=|z|^2/2$.
Let $e_1=(1,0,\ldots,0)\in\C^n$. Take centers $a_j=4je_1$,
radii $\varepsilon_j=2^{-j-4}$, and
\[
 f(z)=\sum_{j=1}^\infty
    \varepsilon_j^{-d/2}\mathbf1_{B(a_j,\varepsilon_j)}(z).
\]
For every $0<q<2$ and $0<s<\infty$,
$M_{q,1}f\in L^s$ and hence $f\in\IDA^{s,q}$: the local $L^q$ size
of the $j$th spike is a constant multiple of
$\varepsilon_j^{d(1/q-1/2)}$. The local $L^1$ masses are
$C\varepsilon_j^{d/2}$, so positive Toeplitz theory gives
$T_f\in\mathcal S_s$ for every $s>0$.
Nevertheless, $M_{2,1}f\ge c>0$ on the disjoint balls
$B(a_j,1/2)$, and therefore $M_{2,1}f\notin L^s$ for every finite $s$.
For each fixed $\zeta$, the normalized Gaussian kernel gives
\[
 \int_{\C^n}|f(w)k_\zeta(w)|^2e^{-|w|^2}\,dv(w)
 \le C_\zeta\sum_{j=1}^\infty e^{-c j^2}<\infty.
\]
Thus the symbol belongs to the kernel $L^2$ domain as well.
Also, $f\in L^1(dv)$ and $|k_\zeta|e^{-\varphi}$ is bounded, so
$f\in\Symb$.

The analytic-distance conditions are different too. Fix
$1\le q<2$. The functions $M_{q,1}f$ are supported on disjoint balls
of uniformly bounded volume, with geometrically decreasing upper
bounds. Thus, for every $0<\rho<\infty$ and $0<\tau\le\infty$,
\[
 M_{q,1}f,\ G_{q,1}f\in L^{\rho,\tau},
 \qquad T_f\in\Sch_{\rho,\tau}.
\]
For the function assertion, use $M_{q,1}f\in L^\infty\cap L^s$
with $0<s<\rho$ and $G_{q,1}f\le M_{q,1}f$. For the operator
assertion, use the already proved membership $T_f\in\Sch_s$ and
$\ell^s\subset\ell^{\rho,\tau}$.

In contrast, $G_{2,1}f\notin L^{\rho,\infty}$ for any finite
$\rho$. Take $z\in B(a_j,1/4)$ and put $D=B(z,1)$ and
$v_d=|B(0,1)|$. The restriction $F=f|_D$ consists of the $j$th
spike, whose support lies in $B(z,1/2)$. Let $\Pi_D$ be the
unweighted Bergman projection on $D$. Holomorphic submean on
$B(w,1/2)\subset D$, for $w\in\operatorname{supp}F$, gives
\[
 \sup_{\operatorname{supp}F}|h|
 \le 2^{d/2}v_d^{-1/2}\|h\|_{L^2(D)}
 \qquad(h\in L^2(D)\cap\mathcal H(D)).
\]
By duality within this Hilbert space,
\[
 \|\Pi_DF\|_{L^2(D)}
 \le 2^{d/2}v_d^{-1/2}\|F\|_{L^1(D)}
 =2^{d/2}v_d^{1/2}\varepsilon_j^{d/2}.
\]
Since $\|F\|_{L^2(D)}^2=v_d$, orthogonal projection yields
\[
 G_{2,1}f(z)^2
 =\frac{\|F\|_{L^2(D)}^2-\|\Pi_DF\|_{L^2(D)}^2}{v_d}
 \ge1-(2\varepsilon_j)^d\ge\tfrac12.
\]
Only square-integrable holomorphic approximants can have finite
$L^2$ error, so this projection identity agrees with the definition
of $G_{2,1}$. The balls $B(a_j,1/4)$ are disjoint and have the same
positive volume. Hence a positive level set of $G_{2,1}f$ has
infinite measure. In particular, for every $0<\rho<\infty$ and
$0<\tau\le\infty$,
\[
 f\in\IDA^{(\rho,\tau),q}\quad(1\le q<2),
 \qquad f\notin\IDA^{(\rho,\infty),2}.
\]
Even membership in the kernel-$L^2$ domain therefore does not allow
the local-$q$ hypothesis to be replaced by a local-$L^2$ IDA
hypothesis. The decomposition in Theorem~\ref{thm:lorentz} handles
this distinction.
\end{example}

\begin{example}[The secondary index in the IDA hypothesis]
\label{ex:secondary-index}
Fix $r>0$, $0<\tau<\rho<\infty$, and
$0<p<\min\{1,\rho\}$. On the classical Fock space
$\varphi(z)=|z|^2/2$, there is $f\in C^\infty(\C^n)\cap L^\infty(\C^n)$ such that
\[
 T_f\in\Sch_p\subset\Sch_{\rho,\tau},\qquad
 G_{2,r}f,\ M_{2,r}f\in L^\rho\setminus L^{\rho,\tau}.
\]
Thus, in Theorem~\ref{thm:lorentz}, replacing the hypothesis
$G_{2,r}f\in L^{\rho,\tau}$ by $G_{2,r}f\in L^\rho$ can invalidate
the equivalence between the operator condition and the local-size
condition. This does not assert optimality among all possible
analytic-distance hypotheses.

Choose a nonzero function $\chi\in C_c^\infty(B(0,r/4))$ with
$0\le\chi\le1$, and put
\[
 u_\lambda(z)=\chi(z)e^{i\lambda\operatorname{Re}z_1},
 \qquad \lambda>0.
\]
We first verify two facts about these symbols.

Let $A:F^2_\varphi\to L^2_\varphi$ be multiplication by
$\chi^{1/2}$. Since $A^*A=T_\chi$, the positive Toeplitz criterion
in Lemma~\ref{lem:positive-lorentz} gives $A\in\Sch_{2p}$.
Indeed, every fixed-radius local mean of $\chi$ is bounded and
compactly supported. Let $U_\lambda$ be multiplication by
$e^{i\lambda\operatorname{Re}z_1}$ on $L^2_\varphi$. Then
\[
 T_{u_\lambda}=A^*U_\lambda A.
\]
The operators $U_\lambda$ converge weakly to zero as
$\lambda\to\infty$. This follows from the Riemann--Lebesgue
lemma, since the product of two $L^2_\varphi$ functions, with the
weight $e^{-2\varphi}$, is integrable.

Choose finite-rank maps $A_N$ converging to $A$ in $\Sch_{2p}$.
For fixed $N$, the finite matrix of $A_N^*U_\lambda A_N$ tends to
zero entrywise, hence in $\Sch_p$. The Schatten product inequality
and the $p$-triangle inequality give, uniformly in $\lambda$,
\[
 \begin{split}
 \|A^*U_\lambda A-A_N^*U_\lambda A_N\|_{\Sch_p}^p
 &\le \|A-A_N\|_{\Sch_{2p}}^p
    \bigl(\|A\|_{\Sch_{2p}}^p+\|A_N\|_{\Sch_{2p}}^p\bigr).
 \end{split}
\]
First let $\lambda\to\infty$ and then $N\to\infty$. We obtain
\begin{equation}\label{eq:oscillatory-bump-small}
 \|T_{u_\lambda}\|_{\Sch_p}\longrightarrow0.
\end{equation}

Put $D=B(0,r/2)$, and let $\Pi_D$ be the unweighted Bergman
projection onto $A^2(D)=L^2(D)\cap\mathcal H(D)$. The map $\Pi_D M_\chi$ on $L^2(D)$
is compact. In fact, it is Hilbert--Schmidt, since
\[
 \int_D |\chi(w)|^2 K_D(w,w)\,dv(w)<\infty,
\]
where $K_D$ is the unweighted Bergman kernel and $\supp\chi$
is a compact subset of $D$. The functions
$e^{i\lambda\operatorname{Re}z_1}$ converge weakly to zero in
$L^2(D)$. Hence $\|\Pi_Du_\lambda\|_{L^2(D)}\to0$, while
$\|u_\lambda\|_{L^2(D)}=\|\chi\|_{L^2(D)}$. There is
$\lambda_0>0$ such that, for every $\lambda\ge\lambda_0$,
\begin{equation}\label{eq:oscillatory-bump-distance}
 \inf_{h\in A^2(D)}\|u_\lambda-h\|_{L^2(D)}^2
 =\|\chi\|_{L^2(D)}^2-\|\Pi_Du_\lambda\|_{L^2(D)}^2
 \ge\tfrac12\|\chi\|_{L^2(D)}^2.
\end{equation}

Choose $a$ with $1/\rho<a\le1/\tau$, and set
\[
 c_j=j^{-1/\rho}[\log(e+j)]^{-a},\qquad
 z_j=4rj e_1,\qquad j\ge1.
\]
Here $e_1=(1,0,\ldots,0)\in\C^n$. By
\eqref{eq:oscillatory-bump-small}, choose $\lambda_j\ge\lambda_0$
so that $\|T_{u_{\lambda_j}}\|_{\Sch_p}\le2^{-j}$. Define
\[
 f(z)=\sum_{j\ge1}c_j u_{\lambda_j}(z-z_j).
\]
The supports are disjoint and locally finite, so $f$ is smooth
and bounded. In particular, it belongs to the initial symbol domain
$\mathcal D_\varphi$.

The classical Fock translation
\[
 (W_bv)(z)=e^{z\cdot\bar b-|b|^2/2}v(z-b)
\]
is unitary on both $L^2_\varphi$ and $F^2_\varphi$. It commutes
with the projection and gives
$T_{u(\cdot-b)}=W_bT_uW_b^*$. Consequently, the series of
Toeplitz operators converges in $\Sch_p$, because
\[
 \sum_{j\ge1}\|c_jT_{u_{\lambda_j}(\cdot-z_j)}\|_{\Sch_p}^p
 \le\sum_{j\ge1}c_j^p2^{-jp}<\infty.
\]
Its sum is the original $T_f$. Indeed, the symbol partial sums
converge uniformly to $f$, with error at most $c_{N+1}$ after
$N$ terms. Their Toeplitz operators therefore converge to $T_f$
in operator norm. This identifies the $\Sch_p$ limit.

For the local functions, the balls $B(z_j,5r/4)$ are disjoint,
and a ball $B(z,r)$ meets at most one symbol support. Thus
\[
 M_{2,r}f(z)\le C\sum_{j\ge1}c_j
         \mathbf1_{B(z_j,5r/4)}(z).
\]
If $z\in B(z_j,r/4)$, then $z_j+D\subset B(z,r)$.
Restriction of a holomorphic approximant to $z_j+D$, followed
by \eqref{eq:oscillatory-bump-distance}, gives
\[
 G_{2,r}f(z)^2
 \ge\frac{c_j^2}{|B(0,r)|}
    \inf_{h\in A^2(D)}\|u_{\lambda_j}-h\|_{L^2(D)}^2
 \ge C_0c_j^2,
\]
where $C_0>0$ is independent of $j$ and $z$. We have proved
\[
 C_1\sum_{j\ge1}c_j\mathbf1_{B(z_j,r/4)}
 \le G_{2,r}f\le M_{2,r}f
 \le C_2\sum_{j\ge1}c_j\mathbf1_{B(z_j,5r/4)}.
\]
Each family consists of disjoint sets of equal positive measure.
The decreasing sequence $(c_j)$ satisfies
\[
 \sum_j c_j^\rho<\infty,
 \qquad
 \sum_j j^{\tau/\rho-1}c_j^\tau
 =\sum_j\frac{1}{j[\log(e+j)]^{a\tau}}=\infty.
\]
The two step functions therefore belong to $L^\rho$ and not to
$L^{\rho,\tau}$. The preceding bounds give the same conclusion
for $G_{2,r}f$ and $M_{2,r}f$. Finally, $p<\rho$ implies
$\Sch_p\subset\Sch_{\rho,\tau}$. This proves the claims.
\end{example}

\subsection*{Scope and further questions}

The mixed mapping results concern $1\le p,q<\infty$.
The Schatten--Lorentz criteria act on $F^2_\varphi$.
All results assume \eqref{eq:weight}.
Theorem~\ref{thm:lorentz} requires $G_{q,r}f\in L^{\rho,\tau}$.
When $\tau<\rho$, the weaker condition $G_{q,r}f\in L^\rho$
need not suffice; see Example~\ref{ex:secondary-index}.
To preserve a leading counting term, we require
$k^{1/\rho}s_k(R)\to0$ for the error operator $R$.
Proposition~\ref{prop:counting-transfer} gives sufficient conditions.
Theorem~\ref{thm:regular-decay} covers decay functions of finite upper
type. It does not cover exponential decay.
The exact Weyl law for strictly convex quadratic weights is proved in
Theorem~\ref{thm:quadratic-weyl}. Proposition~\ref{prop:no-general-weyl}
shows that uniform Hessian bounds alone do not determine a global Weyl
constant. A natural further question is which weaker asymptotic
geometric assumptions suffice.

\begingroup
\renewcommand{\addcontentsline}[3]{}
\section*{Declarations}
\endgroup
\textbf{Funding.} The first author was supported by the National Natural
Science Foundation of China (Grant No.~12401154). The second author was
supported by the National Natural Science Foundation of China
(Grant No.~12601234).

\textbf{Data availability.} No data were used in this study.

\textbf{Competing interests.} The authors declare no competing interests.

\clearpage
\begingroup
\linespread{0.93}\selectfont

\endgroup

\begin{thebibliography}{99}
\bibitem{Bauer} W. Bauer,
Mean oscillation and Hankel operators on the Segal--Bargmann space,
\emph{Integral Equations Operator Theory} \textbf{52} (2005), no.~1, 1--15.

\bibitem{BennettSharpley} C. Bennett and R. Sharpley,
\emph{Interpolation of Operators},
Pure and Applied Mathematics \textbf{129}, Academic Press, Boston, 1988.

\bibitem{BCI} W. Bauer, L. A. Coburn and J. Isralowitz,
Heat flow, BMO, and the compactness of Toeplitz operators,
\emph{J. Funct. Anal.} \textbf{259} (2010), no.~1, 57--78.

\bibitem{CIL} L. A. Coburn, J. Isralowitz and B. Li,
Toeplitz operators with BMO symbols on the Segal--Bargmann space,
\emph{Trans. Amer. Math. Soc.} \textbf{363} (2011), no.~6, 3015--3030.

\bibitem{DaugeRobertWeyl} M. Dauge and D. Robert,
Weyl's formula for a class of pseudodifferential operators with negative order on $L^2(\mathbb R^n)$,
in \emph{Pseudo-Differential Operators (Oberwolfach, 1986)},
Lecture Notes in Math. \textbf{1256}, Springer, Berlin, 1987, pp.~91--122.
\url{https://doi.org/10.1007/BFb0077739}.

\bibitem{Delin} H. Delin,
Pointwise estimates for the weighted Bergman projection kernel in $\mathbb C^n$,
using a weighted $L^2$ estimate for the $\bar\partial$ equation,
\emph{Ann. Inst. Fourier (Grenoble)} \textbf{48} (1998), no.~4, 967--997.

\bibitem{FWZ} Z. Fan, X. Wang and Z. Zeng,
IDA function and asymptotic behavior of singular values of Hankel operators on weighted Bergman spaces,
\emph{Adv. Math.} \textbf{501} (2026), Article 111113.
\url{https://doi.org/10.1016/j.aim.2026.111113}.

\bibitem{Farnsworth} D. Farnsworth,
Hankel operators, the Segal--Bargmann space, and symmetrically-normed ideals,
\emph{J. Funct. Anal.} \textbf{260} (2011), no.~5, 1523--1542.

\bibitem{FollandPhaseSpace} G. B. Folland,
\emph{Harmonic Analysis in Phase Space},
Annals of Mathematics Studies \textbf{122}, Princeton University Press,
Princeton, NJ, 1989.

\bibitem{HaV} R. Hagger and J. A. Virtanen,
Compact Hankel operators with bounded symbols,
\emph{J. Operator Theory} \textbf{86} (2021), no.~2, 317--329.

\bibitem{HL2011} Z. Hu and X. Lv,
Toeplitz operators from one Fock space to another,
\emph{Integral Equations Operator Theory} \textbf{70} (2011), no.~4, 541--559.

\bibitem{HL} Z. Hu and X. Lv,
Toeplitz operators on Fock spaces $F^p(\varphi)$,
\emph{Integral Equations Operator Theory} \textbf{80} (2014), 33--59.


\bibitem{HVTrans} Z. Hu and J. A. Virtanen,
Schatten class Hankel operators on the Segal--Bargmann space and the Berger--Coburn phenomenon,
\emph{Trans. Amer. Math. Soc.} \textbf{375} (2022), no.~5, 3733--3753.
\url{https://doi.org/10.1090/tran/8638}.

\bibitem{HVAPDE} Z. Hu and J. A. Virtanen,
IDA and Hankel operators on Fock spaces,
\emph{Anal. PDE} \textbf{16} (2023), no.~9, 2041--2077.
\url{https://doi.org/10.2140/apde.2023.16.2041}.

\bibitem{HVCorr} Z. Hu and J. A. Virtanen,
Corrigendum to ``Schatten class Hankel operators on the Segal--Bargmann space and the Berger--Coburn phenomenon'',
\emph{Trans. Amer. Math. Soc.} \textbf{376} (2023), no.~8, 6011--6014.
\url{https://doi.org/10.1090/tran/8857}.


\bibitem{HuWang} Z. Hu and E. Wang,
Hankel operators between Fock spaces,
\emph{Integral Equations Operator Theory} \textbf{90} (2018), no.~3, Paper No.~37.

\bibitem{HWZ} W. Huang, X. Wang and Z. Zeng,
Asymptotic behavior of singular values of Toeplitz and Hankel operators on Fock-type spaces,
\emph{Banach J. Math. Anal.} \textbf{20} (2026), Article 16.
\url{https://doi.org/10.1007/s43037-025-00477-8}.

\bibitem{IVW} J. Isralowitz, J. A. Virtanen and L. Wolf,
Schatten class Toeplitz operators on generalized Fock spaces,
\emph{J. Math. Anal. Appl.} \textbf{421} (2015), no.~1, 329--337.
\url{https://doi.org/10.1016/j.jmaa.2014.05.065}.

\bibitem{IZ} J. Isralowitz and K. Zhu,
Toeplitz operators on the Fock space,
\emph{Integral Equations Operator Theory} \textbf{66} (2010), no.~4, 593--611.

\bibitem{JPR} S. Janson, J. Peetre and R. Rochberg,
Hankel forms and the Fock space,
\emph{Rev. Mat. Iberoam.} \textbf{3} (1987), no.~1, 61--138.

\bibitem{Luecking} D. H. Luecking,
Characterizations of certain classes of Hankel operators on the Bergman spaces of the unit disk,
\emph{J. Funct. Anal.} \textbf{110} (1992), no.~2, 247--271.

\bibitem{Orenstein} A. Orenstein,
Toeplitz operators on the Fock space in a symmetrically-normed ideal,
preprint, arXiv:1409.2410v7 [math.FA], 2018.

\bibitem{SV} A. P. Schuster and D. Varolin,
Toeplitz operators and Carleson measures on generalized Bargmann--Fock spaces,
\emph{Integral Equations Operator Theory} \textbf{72} (2012), no.~3, 363--392.

\bibitem{Simon} B. Simon,
\emph{Trace Ideals and Their Applications}, second ed.,
Mathematical Surveys and Monographs \textbf{120},
American Mathematical Society, Providence, RI, 2005.

\bibitem{WangBMOIMO} E. Wang,
Toeplitz operators with BMO and IMO symbols between Fock spaces,
\emph{Arch. Math. (Basel)} \textbf{114} (2020), no.~5, 541--551.

\bibitem{XuDoubling} C. Xu,
Essential norm and Schatten class of Hankel operators on doubling Fock spaces,
\emph{J. Math. Anal. Appl.} \textbf{542} (2025), no.~2, Article 128823.

\bibitem{XuAMS} C. Xu,
Bounded, compact and Schatten classes Hankel operators on weighted Fock spaces,
\emph{Acta Math. Sci.} \textbf{45} (2025), 1529--1554.

\bibitem{ZhangCaoHe} Y. Zhang, G. Cao and L. He,
Toeplitz operators with $\mathrm{IMO}^{s}$ symbols between generalized Fock spaces,
\emph{J. Funct. Spaces} \textbf{2021} (2021), Article ID 8044854.
\end{thebibliography}
\end{document}